\documentclass[12pt]{article}
\usepackage[margin=1in]{geometry}
\usepackage{amsmath}
\usepackage{hyperref}
\usepackage{pdfpages}
\usepackage[utf8]{inputenc}
\usepackage[T1]{fontenc}
\usepackage[english,ngerman]{babel}
\usepackage{geometry} 
\usepackage{listings}
\usepackage{amssymb}
\usepackage{caption}
\usepackage{graphicx} 
\usepackage{amsthm}
\usepackage{setspace}
\usepackage{appendix}
\usepackage{tikz}
\usepackage{float}
\usetikzlibrary{decorations.pathreplacing}
\usetikzlibrary{calc,shapes.geometric}
\usepackage{mathtools}
\DeclarePairedDelimiter\ceil{\lceil}{\rceil}
\DeclarePairedDelimiter\floor{\lfloor}{\rfloor}

\usepackage{tikz}
\usetikzlibrary{arrows.meta,patterns}
\newtheorem{Theorems}{Theorem}[section]
\newtheorem{Lemma}[Theorems]{Lemma}
\newtheorem{Def}[Theorems]{Definition}
\newtheorem{Proposition}[Theorems]{Proposition}
\newtheorem{Cor}[Theorems]{Corollary}
\theoremstyle{definition}
\newtheorem{Example}[Theorems]{Example}
\newtheorem{Question}{Question}
\theoremstyle{remark}
\newtheorem*{Remark}{Remark}
\theoremstyle{plain}
\begin{document}
\selectlanguage{english}
\title{Bootstrap percolation in Ore-type graphs}
\date{\today{}}
\author{%
  Alexandra Wesolek\footnote{ Department of Mathematics,
  Simon Fraser University, Canada. \texttt{agwesole@sfu.ca}.}}
\maketitle
\begin{abstract}
The $r$-neighbour bootstrap process describes an infection process on a graph, where we start with a set of initially infected vertices and an uninfected vertex becomes infected as soon as it has $r$ infected neighbours. An inital set of infected vertices is called percolating if at the end of the bootstrap process all vertices are infected. We give Ore-type conditions that guarantee the existence of a small percolating set of size $l\leq 2r-2$ if the number of vertices $n$ of our graph is sufficiently large: if $l\geq r$ and satisfies $2r \geq l+2 \floor*{\sqrt{2(l-r)+0.25}+2.5}-1$ then there exists a percolating set of size $l$ for every graph in which any two non-adjacent vertices $x$ and $y$ satisfy $\deg(x)+\deg(y)  \geq n+4r-2l-2\floor*{\sqrt{2(l-r)+0.25}+2.5}-1$ and if $l$ is larger with  $l\leq 2r-2$ there exists a percolating set of size $l$ if $\deg(x)+\deg(y) \geq n+2r-l-2$. Our results extend the work of Gunderson, who showed that a graph with minimum degree $\lfloor n/2 \rfloor+r-3$ has a percolating set of size $r \geq 4$. We also give bounds for arbitrarily large $l$ in the minimum degree setting.

\end{abstract}
\section{Introduction}
Bootstrap percolation models the spread of an infection over a graph. In the $r$-neighbour
bootstrap process on a graph $G$, we start with a set of infected vertices $A_0 \subset V(G)$ and
a new vertex gets infected as soon as it has $r$ infected neighbours. If we think of it in rounds of infection, we get 
\begin{align*}
A_t=A_{t-1} \cup \{ v \in V(G) \vert \ \vert N(v) \cap A_{t-1} \vert \geq r \},
\end{align*} where, for all time steps $t\in \mathbb{N}$, $A_t$ is the set of infected vertices at this time step. In bootstrap percolation one is interested in initially infected sets $A_0$ which have the property
that at the end of the process all vertices are infected. 
\begin{Def}
Given a graph $G$, a set $A_0\subset V(G)$ is called $r$-percolating if there exists a time step $t\in \mathbb{N}$ in the $r$-neighbour bootstrap process such that $A_t=V(G)$.
\end{Def}
The motivation to look at bootstrap processes came originally from a problem in physics when Chalupa, Leith and Reich \cite{chalupa1979bootstrap} looked at lattices and bootstrap processes as a model of ferromagnetism. But bootstrap percolation has many applications, for example in epidemiology  \cite{roberts2003challenges} or in business marketing \cite{chen2009approximability, kempe2003maximizing,nichterlein2010tractable}. These problems are naturally probabilistic and started the research in bootstrap percolation on random graphs  \cite{amini2012tell, balogh2007bootstrap, janson2012bootstrap} or on fixed graphs but where one starts with a randomly chosen set $A_0$ \cite{aizenman1988metastability,holroyd2003sharp,balogh2006bootstraphypercube, balogh2012sharp, balogh2006bootstrap}. Later research was more concerned with extremal problems in bootstrap percolation. Given some graph $G$, one is particularly interested in the size of the smallest $r$-percolating set which is commonly denoted by $m(G,r)$.  The first such result concerned $[n]^d$ the $d$-dimensional grid on $n^d$ vertices where Balogh and Pete \cite{balogh1998random} determined the case when the bootsrap threshold, the number of neighbours a vertex needs to get infected, is $r=d$. 
\begin{Theorems}[Balogh and Pete, 1998]
\label{baloghpete}
For all $n,d \in \mathbb{N}$
\begin{align*}
 m([n]^d,d)=n^{d-1}.
\end{align*} 
\end{Theorems}
Balogh, Bollobás and Morris \cite{balogh2010bootstrap} gave results in the case when the bootstrap threshold is $r=2$ by showing that $m([n]^d,2)=\ceil*{\frac{d(n-1)}{2}}+1$. For $n=2$ this result nicely complements Theorem \ref{baloghpete} as only the cases $2 \leq r \leq d $ are not trivial. For higher values of $r$, answering a conjecture in \cite{balogh1998random}, Huang and Lee \cite{huang2013deterministic} determined asymptotically the minimum size of $r$-percolating sets in $[n]^d$ for $n \to \infty$ and fixed $d$ such that $d+1\leq r \leq 2d$ which is $(1-\frac{d}{r})n^d+O(n^{d-1})$. For $3\leq r \leq d-1$ there is in general no asymptotically tight bound known. Recently, Morrison and Noel looked at a slightly different problem where they fixed $n=2$ and determined $m([2]^d,r)=m(Q_d,r)$ for $d \to \infty$, confirming a conjecture from Balogh and Bollobás \cite{balogh2006bootstraphypercube}.
\begin{Theorems}[Morrison and Noel, 2017] Let $Q_d$ be the $d$-dimensional cube and $r\geq 3$. Then $m(Q_d, r)= \frac{1+o(1)}{r} {d \choose r-1}$ for $d \to \infty$.
\label{Hypercubethm}
\end{Theorems}  They gave an algebraic proof which was then nicely simplified by Hambardzumyan, Hatami and Qian  \cite{hambardzumyan2017polynomial} using a polynomial method.

In this paper we are interested in a different extremal problem, where we fix a maximum size for the initially infected set $A_0$ and determine which graphs have a percolating set of this size. More specifically, we are interested in properties for a graph $G$ that guarantee that $m(G,r)$ is low. The first results of this type connected $m(G,r)$ to the degree sequence of $G$  \cite{ackerman2010combinatorial,reichman2012new}, where Reichman gave the following upper bound.
\begin{Theorems}[Reichman, 2012]
Let $G$ be a graph and $\deg(v)$ the degree of $v\in V(G)$. Then
\begin{align*}
m(G,r) \leq \sum_{v \in V(G)} \min \left\{ 1, \frac{r}{\deg(v)+1} \right\}.
\end{align*}
\end{Theorems}
This shows that the denser our graph is the easier it is to infect every vertex. Note that determining whether $m(G,r) \leq r-1$ is trivial, as this holds exactly $G$ has at most $r-1$ vertices. So, to ask whether $m(G,r)=r$ is the first non-trivial question. Freund, Poloczek and Reichman \cite{freund2018contagious} were interested in how many edges a graph $G$ on $n$ vertices needs to have to guarantee a percolating set of size $r$, and they showed:
\begin{Theorems}[Freund, Poloczek and Reichman, 2015]
Let $G$ be a graph on $n$ vertices with $n\geq 2r+2$ and $e(G)\geq {n-1 \choose 2}+1$. Then $G$ has a percolating set of size $r$.
\end{Theorems}
This result is tight since the graph $G$ consisting of a clique on $n-1$ vertices and another isolated vertex has $e(G)={n-1 \choose 2}$ but has no percolating set of size $r$. Note that, for fixed $n$, this result does not depend on $r$ as long as $n\geq 2r+2$. 

Moreover, Freund, Poloczek and Reichman, as well as Gunderson \cite{gunderson2017minimum}, were interested in how large the minimum degree of a graph needs to be to guarantee a percolating set of size $r$. It is clear that any $r$ vertices percolate in one time step if our minimum degree is $n-1$, i.e.\! our graph is a clique. Gunderson proved a lemma that showed we can decrease the minimum degree by essentially a fraction of $n$ and still any $r$ vertices percolate.
\begin{Lemma}
\label{Lemma1}
If $G$ is a graph on $n$ vertices with $\delta(G) \geq n-\floor*{\frac{n+1}{r+1}}$ and $A_0 \subset V(G)$ is a set with $r$ vertices, then $A_0$ percolates.
\end{Lemma} 
Gunderson did not determine whether this bound on $\delta(G)$ can be improved. She was interested in a problem with weaker conditions, where we only require one percolating set of size $r$. Freund, Poloczek and Reichman gave the first results in relation to this problem.
\begin{Theorems}[Freund, Poloczek and Reichman, 2015]
\label{Freundpoloczekreichman}
Let $r\geq 2$ and $G$ be a graph on $n$ vertices. If $\delta(G)\geq \ceil*{\frac{r-1}{r}n}$, then $G$ has a percolating set of size $r$.
\end{Theorems}
We will show in Section \ref{ChapterExtensionGunderson} that, when $n$ is not divisible by $r$, then in fact any $r$ vertices percolate in a graph with minimum degree at least $\ceil*{\frac{r-1}{r}n}$, while if $n$ is divisible by $r$ then any $r$ vertices percolate in a graph with minimum degree $\frac{r-1}{r}n+1$. For odd $r \geq 3$ the graph consisting of a clique on $r+1$ vertices with a perfect matching deleted has $\delta(G)=r-1=\floor*{(r-1)\frac{r+1}{r}}$ and no percolating set of size $r$, showing that the above result is tight for $n=r+1$. Since this example has only a few vertices with respect to $r$, Gunderson asked whether for $n=n(r)$ large enough the bound can be improved. It is clear that one needs a minimum degree of at least $\floor*{\frac{n}{2}}$, as the graph consisting of two disjoint cliques, one of size $\floor*{\frac{n}{2}}$ and one of size $\ceil*{\frac{n}{2}}$, has minimum degree $\floor*{\frac{n}{2}}-1$ but is disconnected and so has no percolating set of size $r$. For $r=1$ it is actually sufficient to have a minimum degree of $\floor*{\frac{n}{2}}$, since as soon as a graph is connected we eventually infect all vertices if we start with $r=1$ vertices. In fact, the same bound on the minimum degree guarantees the existence of a percolating set of size $r=2$; a fact which follows from ideas of Freund, Poloczek and Reichman if one does one extra check or which is seen directly from later work of Dairyko et al. \cite{dairyko2016ore}. Gunderson gave results for $r\geq 3$ which show that in the general case the maximum degree of a graph with no percolating set of size $r$ is still roughly $\floor*{\frac{n}{2}}$. This is substantially different from the first bound of Theorem \ref{Freundpoloczekreichman}.
\begin{Theorems}[Gunderson, 2017] If $r=3$ and $n \geq 30$, any graph $G$ on $n$ vertices with $\delta(G) \geq \floor*{\frac{n}{2}}+1$ satisfies $m(G,3) = 3$.
\end{Theorems}
\begin{Theorems}[Gunderson, 2017] For any $r \geq 4$ and $n$ sufficiently large, if $G$ is a graph on $n$
vertices with $\delta(G) \geq \floor*{\frac{n}{2}}+r-3$, then $m(G,r) = r$. 
\end{Theorems}
Gunderson showed that the above bounds on $\delta(G)$ are tight by giving an example of a family of graphs which have minimum degree $\max\{\floor*{\frac{n}{2}}+r-4,\floor*{\frac{n}{2}}\}$ and no percolating set of size $r \geq 3$; these are essentially two disconnected cliques with a sparse regular bipartite graph between them. Having solved this problem, she asked how much we can weaken the minimum degree conditions if we start with an initially infected set $A_0$ of size $l>r$.
\begin{Question} For $n=n(r,l)$ large enough and fixed $l>r$, how big does the minimum degree of a graph need to be in order to guarantee the existence of an $r$-percolating set of size $l$?
\end{Question}
More precisely, she asked about the value of the following parameter for large enough $n$.
\begin{Def}
Let $n,r,l\in \mathbb{N}$, then we define $\delta_0(n,r,l)$ to be the minimum number such that any graph $G$ on $n$ vertices with minimum degree $\delta(G)\geq \delta_0(n,r,l)$ has an $r$-percolating set of size $l$. 
\end{Def}
Another extension of Gunderson's results asks about percolating sets in Ore-type graphs. Before introducing Ore-type graphs we want to say something about Ore graphs. These are graphs where any two non-adjacent vertices $x$ and $y$ satisfy $\deg(x)+\deg(y)\geq n$. Their name comes from a famous theorem of Ore \cite{ore1960note} which is a generalization of Dirac's Theorem about Hamiltonicity \cite{dirac1952some}. 
\begin{Theorems}[Dirac, 1952]
\label{Dirac}
If $G$ is a graph with $\delta(G)\geq \frac{n}{2}$, then $G$ is Hamiltonian.
\end{Theorems}
Note that any graph $G$ with $\delta(G)\geq \frac{n}{2}$ is also an Ore graph. Therefore the following theorem by Ore generalizes Theorem \ref{Dirac}.
\begin{Theorems}[Ore, 1960]
Every Ore graph has a Hamiltonian cycle.
\end{Theorems}
Gunderson asked if there is a similar extension of her results on $\delta_0$. For this we need to introduce Ore-type graphs in which the sum of the degrees of two non-adjacent vertices is bounded from below.
\begin{Def}
Let $G$ be a graph, then
\begin{align*}
D(G) \coloneqq \min_{v\neq w \in V(G), v \nsim w} \deg(v)+\deg(w).
\end{align*}
\end{Def}
Freund, Poloczek and Reichman showed that any Ore graph has a percolating set of size $r=2$. This is not a direct generalization of the minimum degree result $\delta_0(n,2,2)=\floor*{\frac{n}{2}}$ as if $n$ is odd and $\delta(G)=\floor*{\frac{n}{2}}$ we do not automatically have $D(G)\geq n$. Dairyko et al. gave a proper extension later by showing that if $G$ is not $C_5$, the cycle on $5$ vertices, then whenever $D(G)\geq n-1$ we have a percolating set of size $r=2$. Note that this is a proper extension of the minimum degree result. There have been no results so far in the case when $r\geq 3$, which may have motivated the following second question.
\begin{Question} For $r\in \mathbb{N}$ and $n=n(r)$ large enough, how big does $D(G)$ need to be such that we can guarantee the existence of an $r$-percolating set of size $r$?
\end{Question}
In the first part of this paper we answer Question $1$ and Question $2$ at once for small enough $l=l(r)$. For this we define 
\begin{align*}
f(k)\coloneqq \floor*{\sqrt{2k+0.25}+2.5},
\end{align*}
so that $f(k)$ is the largest number satisfying the equation ${f(k)-2 \choose 2} \leq k$. 

We prove the following theorem in Section \ref{chapter5enough}.
\begin{Theorems}
\label{mainthm}
Let $l\geq r$ and $2r \geq l+2f(l-r)-1$. For sufficiently large $n$, any $n$-vertex graph $G$ with $D(G) \geq n+ 4r-2l-2f(l-r)-1$ has a percolating set of size $l$. 
\end{Theorems}
Note that this also extends Gunderson's result to $l \geq r$ since if our graph $G$ has degree $\delta(G)\geq \floor*{\frac{n}{2}}+2r-l-f(l-r)$ then for any two vertices $x$ and $y$ in $G$ we have that the sum of their degrees $\deg(x)+\deg(y) \geq n+ 4r-2l-2f(l-r)-1$. In particular, this implies for $l=r$ and $r$ large enough that $\delta_0(n,r,r)\leq \floor*{\frac{n}{2}}+r-3$. Moreover, we give examples to show that our result is tight when $3r\geq  2l +f(l-r)+4$ in both the Ore-type setting and the minimum degree setting, answering the open questions mentioned by Gunderson for small enough $l=l(r)$. Additionally, for the generalization of Question $2$, we will determine results where $l$ is closer to $2r$. For this we define the following parameter.
\begin{Def} For $n,r,l\in \mathbb{N}$ let $D_0(n,r,l)$ be the smallest natural number $D$ such that any graph $G$ on $n$ vertices with $D(G)\geq D$ has a percolating set of size $l$. If $l=r$, we write $D_0(n,r)=D_0(n,r,r)$.
\end{Def}
We determine $D_0(n,r)$ under some divisibility conditions on $n$ in Section \ref{resultsforbigk}.
\begin{Cor}
\label{Corollarykiszero}
Given $r\geq 1$, let $n=n(r)$ be sufficiently large. For $r\notin \{1,2,4\}$ let $n$ be even and for $r=4$ let $n$ be divisible by $3$, then
\begin{align*}
D_0(n,r)= \begin{cases} 
      n+2r-7 & \text{for } r\geq 5,\\ 
     n+ r-2 & \text{for } r \in \{3,4\}, \\ 
     n-1 & \text{for } r \in \{1,2\}.\\
   \end{cases} 
\end{align*}
\end{Cor}
If $n$ does not satisfy the divisibility conditions, we get an almost tight bound as we have the same upper bound but the lower bound is slightly smaller. For example, for $r\geq 5$ and $n$ odd we take the construction of $n+1$ vertices of the graph $G$ with $D(G)= D_0(n+1,r)-1$ which does not have a percolating set of size $r$ and delete a vertex. Note that this graph $\tilde{G}$ has $n$ vertices and $D(\tilde{G}) \geq D_0(n+1,r)-3$ and it will be easy to see that this graph $\tilde{G}$ also does not have a percolating set of size $r$.
Moreover, we extend Gunderson's Lemma \ref{Lemma1} not only to percolating sets of size $2r-2 \geq l>r$, but also improve her result when $l=r$.
\begin{Lemma}
\label{lemmaimprovedgunderson}
Let $2r-2\geq l \geq r$. If $G$ is a graph with minimum degree $\delta(G)\geq n-\ceil*{\frac{l-r+1}{l} n}+(l-r)+1$ and $A_0 \subset V(G)$ with $\vert A_0 \vert=l$, then $A_0$ percolates.
\end{Lemma}
This lemma is proved in Section \ref{ChapterExtensionGunderson} and we show that the result is tight if $n\geq 2l$.\\

We use the following notation. Unless stated otherwise, the variable $n=v(G)$ will always denote the number of vertices of a graph $G$. For a vertex $v \in V(G)$ and $A\subset V(G)$ the characteristic function $1_A(v)$ equals $1$ if $v \in A$ and $0$ otherwise. For a set $W \subset V(G)$ we let $N_W(v)$ be the set of neighbours of $v$ in $W$, whereas $N(v)=N_{V(G)}(v)$. Moreover, let 
\begin{align*}
N_W(v_1,\dots,v_j)=\bigcup\limits_{i=1}^j N_W(v_i)
\end{align*}
 be the set of vertices in $W$ which have at least one neighbour in $\{v_1,\dots,v_j\}\subset V(G)$. In particular, $N(v_1,\dots,v_j)=N_{V(G)}(v_1,\dots,v_j)$. Whereas $\deg(v)$ is the degree of a vertex $v\in V(G)$, we write for a subset $W \subset V(G)$ 
\begin{align*}
\deg_W(v)=\vert N_W(v)\vert.
\end{align*}
We say $f(n)=O(g(n))$ if there exists a constant $C$ such that $\vert f(n) \vert \leq C \, g(n)$ and  $f(n)=\Theta(g(n))$ if there exist constants $c$ and $C$ such that $c \, g(n)\leq \vert f(n) \vert \leq C \, g(n)$. 
When we say "almost each vertex in $V(G)$", we mean all except $O(1)$ vertices in $V(G)$ for $n  \to \infty$. Moreover, the variables $l$ and $r$ are positive integers.
\section{Minimum degree conditions on graphs for which any $l$ infected vertices percolate}
\label{ChapterExtensionGunderson}
In this section we determine for $l\leq 2r-2$ what minimum degree a graph has to have such that any initially infected set of $l$ vertices percolates. The following result is not only an extension but also an improvement of Gunderson's Lemma \ref{Lemma1}, and we will show that our result cannot be improved.
\begin{Def}The \textbf{closure} of an infected set $A$ is 
\begin{align*}
\langle A \rangle= \cup_{t\geq0} \, A_t,
\end{align*}
where $A_0$ is taken to be $A$. The closure of $A$ is therefore the set of infected vertices at the end of the bootstrap process if we started with $A$ as the initially infected set.
\end{Def}
For a graph $G$ that means that $A_0\subset V(G)$ percolates if and only if $\langle A_0 \rangle=V(G)$.
\begin{Def}
$A$ is called \textbf{closed} if $\langle A \rangle=A$.
\end{Def}
Note that the closure of a set $\langle A_0 \rangle$ is always closed.
\begingroup
\def\theLemma{\ref{lemmaimprovedgunderson}}
\begin{Lemma} Let $k \in \mathbb{N}_0$ and $2r-2\geq l$. If $G$ is a graph with minimum degree $\delta(G)\geq\floor*{\frac{r-1}{l} n}+l-r+1$ and $A_0 \subset V(G)$ with $\vert A_0 \vert=l$, then $A_0$ percolates.
\end{Lemma}
\addtocounter{Theorems}{-1}
\endgroup
\begin{proof}
Suppose for contradiction there exists a set $A_0$ of size $l$ in $G$ that does not percolate. Let $A \coloneqq \langle A_0 \rangle$ and let $a\coloneqq \vert A \vert$. Note that every vertex in $A^c$ has at most $r-1$ neighbours in $A$, as it would get infected otherwise.  Thus every vertex in $A^c$ has degree at most $n-a-1+r-1$ and therefore, by the minimum degree conditions, $ a\leq \ceil*{\frac{l-r+1}{l}n}+2r-l-3$. Furthermore there are at most $(n-a)(r-1)=\vert A^c \vert(r-1)$ edges between $A$ and $A^c$, and by averaging there must be a vertex $x \in A$ with at most $(n-a)(r-1)/a=\frac{r-1}{a}n-(r-1)$ neighbours in $A^c$. Such an $x$ must have degree at most 
\begin{align*}
d(a)\coloneqq(a-1)+\frac{r-1}{a}n-(r-1)=\frac{r-1}{a}n+a-r.
\end{align*}
We show that $d(a)$ is small. By taking the derivative $d'(a)=-\frac{r-1}{a^2} \, n+1$
 it is easy to see that $d'(a)<0$ for $0<a < \sqrt{n(r-1)}$ and $d'(a)>0$ for $a> \sqrt{n(r-1)}$ and therefore $d$ has a minimum at $\sqrt{n(r-1)}$. $d(a)$ therefore takes its maximum at either $a=l$ or $a=\ceil*{\frac{l-r+1}{l}n}+2r-l-3$. But
\begin{align*}
d(l)=\frac{r-1}{l}n+l-r
\end{align*}
and note that $d(l)=d(n(r-1)/l)$ and for $2r-2\geq l$ we have
\begin{align*}
\ceil*{\frac{l-r+1}{l} n}+2r-l-3 & \leq \frac{l-r+1}{l} n+2r-l-2\\
&= \frac{r-1}{l} n-\frac{2r-l-2}{l}(n-l)\\
& \leq \frac{r-1}{l} n
\end{align*}
and therefore we know that $d(a)$ is maximized at $a=l$, which means 
\begin{align*}
d(a)\leq d(l)= \frac{r-1}{l}n+l-r
\end{align*}
and there exists a vertex in $A$ that has degree at most $\floor*{\frac{r-1}{l}n+l-r}$ which is a contradiction to the minimum degree condition of $G$.
\end{proof}
The following example shows that we cannot improve our above result if $n\geq 2l$.
\begin{Example}
We construct a graph $G$ with minimum degree $\delta(G)= \floor*{\frac{r-1}{l}n+l-r}$ for $n\geq 2l$ and a set $A_0\subset V(G)$ of size $l$ which does not percolate. Let $G$ consist of a clique $U$ of $l$ vertices and a clique $W$ of $n-l$ vertices such that every vertex in $W$ has exactly $r-1$ neighbours in $U$ while every vertex in $U$ has either $\floor*{\frac{r-1}{l}n-(r-1)}$ or $\ceil*{\frac{r-1}{l}n-(r-1)}$  neighbours in $W$. An explicit graph for the case when $l=r$ is depicted in Figure \ref{Bild1}.  Let $A_0=V(U)$. Since no vertex of $W$ has $r$ neighbours in $U$, the infection cannot spread from $U$, and so $A_0$ does not percolate. Each vertex in $U$ has degree at least $\floor*{\frac{r-1}{l}n}+l-r$ and since $n \geq 2l$ the degree of each vertex in $W$ is $n-1-(l-r+1) = \floor*{\frac{r-1}{l}n}+ \ceil*{\frac{l-r+1}{l}n}-(l-r)-2\geq \floor*{\frac{r-1}{l}n}+l-r $.
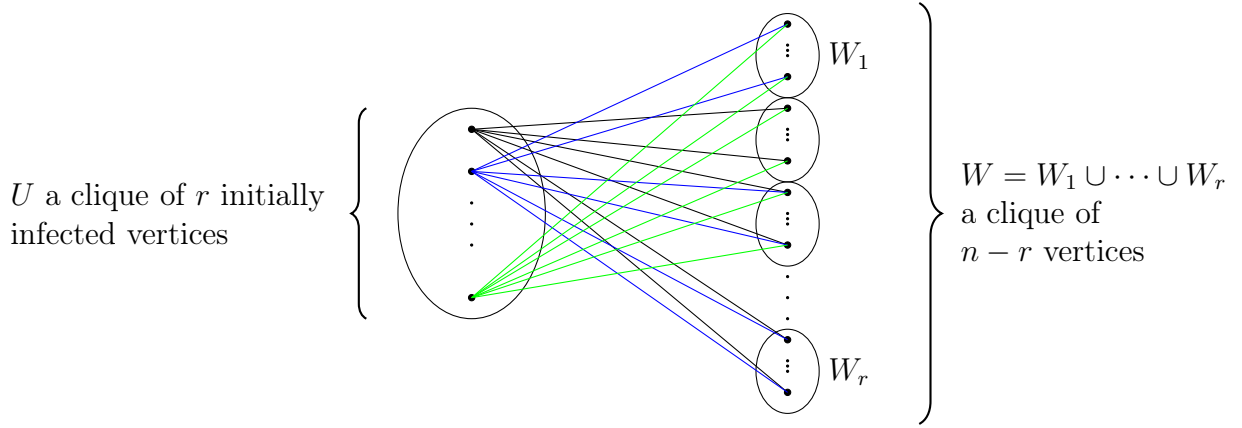
\begin{figure}[h]
\begin{tikzpicture}[scale=1.4]
    \draw (0,1.5) ellipse (0.7cm and 1cm);
    \draw (3,3) ellipse (0.3cm and 0.4cm) node[right,xshift=0.4cm]{$W_1$};
    \draw[thick,black,decorate,decoration={brace,amplitude=0.2cm}] (-1,0.5) -- (-1,2.5) node[midway, left,xshift=-0.4cm,align=left]{$U$ a clique of $r$ initially \\ infected vertices};
        \draw[thick,black,decorate,decoration={brace,amplitude=0.3cm}] (4.25,3.5) -- (4.25,-0.5) node[midway, right,xshift=0.4cm,align=left]{$W=W_1 \cup \dots \cup W_{r}$\\ a clique of \\ $n-r$  vertices};
    \draw (3,2.2) ellipse (0.3cm and 0.4cm);
    \draw (3,1.4) ellipse (0.3cm and 0.4cm);
    \draw (3,0) ellipse (0.3cm and 0.4cm) node[right,xshift=0.4cm]{$W_{r}$};
    \foreach \i in {3,2.2,1.4,0}
   { \draw [fill] (3,\i+0.3) circle [radius=0.03];
    \draw [fill] (3,\i+0.8-1) circle [radius=0.03];
    \draw [fill](3,\i +0.1) circle [radius=0.01];
    \draw [fill](3,\i +0.05) circle [radius=0.01];
    \draw [fill](3,\i ) circle [radius=0.01];
    }
   
   \draw [fill] (0,2*0.4+1.5 ) circle [radius=0.03];
   \foreach \i in {2.2,1.4,0}
   {
   \draw (0,2*0.4+1.5 ) -- (3,\i+ 0.3);
   \draw (0,2*0.4+1.5 ) -- (3,\i+0.8-1);
    }
       
	\draw [fill] (0,2*0.2+1.5 ) circle [radius=0.03];       
        \foreach \i in {3,1.4,0}
   {
   \draw[blue] (0,2*0.2+1.5 ) -- (3,\i+ 0.3);
   \draw[blue] (0,2*0.2+1.5 ) -- (3,\i+0.8-1);
    }    
   
    \draw [fill](0,0.1+1.5 ) circle [radius=0.01];
    \draw [fill](0,-0.1+1.5 ) circle [radius=0.01];
    \draw [fill](0,-0.3+1.5 ) circle [radius=0.01];
    \draw [fill](3,0.1+0.8 ) circle [radius=0.01];
    \draw [fill](3,-0.1+0.8 ) circle [radius=0.01];
    \draw [fill](3,-0.3+0.8 ) circle [radius=0.01];
    
   \draw [fill] (0,-2*0.4+1.5 ) circle [radius=0.03];
           \foreach \i in {3,2.2,1.4}
   {
   \draw[green] (0,-2*0.4+1.5 ) -- (3,\i+ 0.3);
   \draw[green] (0,-2*0.4+1.5 ) -- (3,\i+0.8-1);
    }    

\end{tikzpicture}
\caption{An example of a graph $G$ and a set of $r$ initially infected vertices which is closed. $W_1, \dots,W_r$ are of size $\floor*{\frac{n-r}{r}}$ or $\ceil*{\frac{n-r}{r}}$  each and the $i$-th initially infected vertex is connected to those vertices in $W$ which are not in $W_i$.  }
\label{Bild1} 
\end{figure}
\end{Example}
\section{An asymptotic bound for minimum degree conditions}
We give a bound for large values of $l$ which shows that $\delta_0(n,r,l)$ is essentially changing if $l$ is a multiple of $r$. 
\begin{Theorems}
\label{ThmAsymptotic}
If $l\geq r$, then for $n$ sufficiently large
\begin{align*}
 \floor*{\frac{n}{\floor*{l/r}+1}}\leq \delta_0(n,r,l) \leq \ceil*{\frac{n}{\floor*{l/r}+1}}+\floor*{l/r}(r-1)-1.
\end{align*}
\end{Theorems}
\begin{proof}
For a lower bound we take $\floor*{l/r}+1$ disjoint cliques on roughly $n/(\floor*{l/r}+1)$ vertices. For a clique to get infected it needs to have at least $r$ infected vertices in the beginning. Therefore to have an initial set that is percolating it would need to consist of at least $r(\floor*{l/r}+1)>l$ vertices and therefore the graph does not have a percolating set of size $l$. For the upper bound, we prove by induction on $k$
\begin{align*}
\delta_0(n,r,k \cdot r) \leq \ceil*{\frac{n}{k+1}}+k(r-1)-1
\end{align*}
which then gives us the same upper bound for any $l\geq k \cdot r$.

The base case $k=1$ is an implication of Gunderson's result. Now suppose $k\geq 2$. By infecting $r$ vertices from a $K_{r,k r}$, a complete bipartite graphs on parts of size $r$ and $kr$, which exists for $n$ large enough by  Kovári, Sós and Turán \cite{kovari1954problem}, we can infect at least $(k+1)r$ vertices.  To look at the number of vertices that are infected by those $(k+1)r$ vertices, we consider the closure $A$ and note that by double counting the edges between $A$ and $A^c$ we have
\begin{align*}
(r-1)\vert A^c \vert \geq \left\vert \{ \textrm{edges between }A^c \textrm{ and } A\} \right\vert \geq \vert A \vert \left(\ceil*{\frac{n}{k+1}}+ k(r-1)-1- (\vert A \vert-1)\right) 
\end{align*}
which, by rearranging and noting that $\vert A^c \vert=n-\vert A \vert$, gives
\begin{align*}
D(\vert A \vert)\coloneqq\vert A \vert^2-\vert A \vert \left(\ceil*{\frac{n}{k+1}}+(k+1)(r-1)\right)+n(r-1)\geq 0
\end{align*}
and taking $\vert A \vert=((k+1)(r-1)-1)$ we get
\begin{align*}
&((k+1)(r-1)-1)^2- ((k+1)(r-1)-1) \left(\ceil*{\frac{n}{k+1}}+(k+1)(r-1)\right)+n(r-1)\\
&= -((k+1)(r-1)-1)\left(\ceil*{\frac{n}{k+1}}+1\right)+n(r-1)\\
&\geq  \frac{n}{k+1}- 2((k+1)(r-1)-1)> 0
\end{align*}
for $n$ large enough and taking $\vert A \vert=(k+1)(r-1)$ we get
\begin{align*}
&((k+1)(r-1))^2-(k+1)(r-1) \left(\ceil*{\frac{n}{k+1}}+(k+1)(r-1)\right)+n(r-1)\\
&=-((k+1)(r-1))\ceil*{\frac{n}{k+1}}+n(r-1) \\
&\leq 0. 
\end{align*}
Note that by symmetry of $D$ around $\vert A \vert=\frac{\ceil*{\frac{n}{k+1}}+(k+1)(r-1)}{2}$ we have the same inequalities for $\vert A \vert=\ceil*{\frac{n}{k+1}}+1$ and $\vert A \vert=\ceil*{\frac{n}{k+1}}$. Therefore if $D(\vert A \vert)\geq 0$ then we know $\vert A \vert \leq(k+1)(r-1)$ or $\vert A \vert \geq \ceil*{\frac{n}{k+1}}$ as $D$ is a quadratic polynomial. Now we look at $A^c$ and we know that the remaining vertices have at most $r-1$ neighbours in $A$. Deleting $A$ from the graph results in a graph of degree $\ceil*{\frac{n}{k+1}}+(k-1)(r-1)-1 \geq \frac{1}{k} (\frac{kn}{k+1})+(k-1)(r-1)-1\geq \frac{1}{k} (\vert A^c \vert)+(k-1)(r-1)-1$. This means the remaining graph has a percolating set of size $(k-1)\cdot r$ by induction. Therefore $G$ has a percolating set of size $kr$.
\end{proof}

\section{Percolating sets of size $l$ in Ore-type graphs}
\label{oretypeconditions}
In Theorem \ref{mainthm} we show that for $2r+1\geq l+2f(l-r)$ and a graph $G$ with $D(G)\geq n+ 4r-2l-2f(l-r)-1$ there exists a percolating set of size $l$, where 
\begin{align*}
f(k)\coloneqq \floor*{\sqrt{2k+0.25}+2.5}.
\end{align*}
Since $f(0)=3$ this implies that for $l=r$ and $r\geq 5$ and a graph $G$ with $D(G) \geq n+2r-7$ we can always find a percolating set of size $r$. Similarly as in Ore's extension of Dirac's Theorem, our result implies the upper bound of Gunderson's result if $r\geq 5$ as she showed that if $\delta(G) \geq \floor*{\frac{n}{2}}+r-3$ we always find a percolating set of size $r$.

Whereas $f(k)$ looks rather complicated if one determines it explicitly, one can equally define $f(k)$ to be the largest natural number such that ${f(k) - 2 \choose 2} \leq k$. This means ${f(k) - 1 \choose 2} > k$.
\begin{Proposition}
For $k \in \mathbb{N}$ the function $f(k)= \floor*{\sqrt{2k+0.25}+2.5}$ satisfies $(f(k) - 2)(f(k) - 3) \leq 2k$ and $(f(k) - 1)(f(k) - 2) > 2k$.
\end{Proposition}
\begin{proof}
If we write out $(f(k) - 2)(f(k) - 3)$ we get
\begin{align*}
(f(k) - 2)(f(k) - 3) &=\floor*{\sqrt{2k+0.25}+0.5}\floor*{\sqrt{2k+0.25}-0.5}\\
& \leq \left(\sqrt{2k+0.25}+0.5\right)\left(\sqrt{2k+0.25}-0.5\right) =2k.
\end{align*}
Furthermore for $(f(k) - 1)(f(k) - 2) $ we calculate
\begin{align*}
(f(k) - 1)(f(k) - 2)&=\left(\floor*{\sqrt{2k+0.25}+1.5}\right)\left(\floor*{\sqrt{2k+0.25}+0.5}\right)\\
&>\left(\sqrt{2k+0.25}+0.5\right)\left(\sqrt{2k+0.25}-0.5\right)=2k,
\end{align*}
which proves the lemma.
\end{proof}

\subsection{Tightness results}
Before we show that every $n$-vertex graph $G$ with $D(G)\geq n+ 4r-2l-2f(l-r)-1$ contains a percolating set of size $l$ we will show that our result is tight for $3r\geq 2l +f(l-r)+4$ and $n$ even. The following graph has a minimum degree of $\delta(G)=\frac{n}{2}+ 2r - l - f(l-r)-1$ and no percolating set of size $l$. Our graph is very similar to the graph Gunderson used to show tightness in the case when $l=r$ and in this case ours is a special version of Gunderson's construction.
 \begin{Theorems}
\label{ThmTightness}
Given $l \geq r$ such that $3r \geq 2l +f(l-r)+4$, let $n$ be sufficiently large. Let $H$ be a $(2r-l-f(l-r))$-regular bipartite graph with $2n$ vertices and girth at least $2f(l-r)+2$. The graph $G$ obtained from $H$ by adding all edges inside each part of $H$ has $m(G,r) > l$.
\end{Theorems}
We prove the existence of such a graph $G$ in Proposition \ref{PropTightness}. In order to prove this theorem we would like to introduce a notion about the infected neighbours of a vertex.
\begin{Def}
We define for a vertex $v\in V(G)$ that was not initially infected and a set $W\subset V(G)$ the set $I_{W}(w)$ to be the infected neighbours of $w$ in $W$ at the moment of the infection of $w$. Note that $\vert I_{V(G)}(w) \vert \geq r$ otherwise $w$ would not get infected.
\end{Def}
In order to prove that $G$ has no percolating set of size $l$ we want to think of the subgraph $H$ in the following way.
\begin{Lemma}
A bipartite graph $H$ on parts $U$ and $W$ has girth at least $2g+2$ for $g\in \mathbb{N}$ if and only if for every subset $\{ u_1,\dots, u_j \} \subset U$ with $j\leq g$
\begin{align*}
\vert N(u_1,\dots,u_j) \vert\geq \left( \sum_{i=1}^{j}  \vert N(u_i) \vert \right) -(j-1).
\end{align*}
\label{Lemmatightness}
\end{Lemma}
A proof of the lemma is provided in the Appendix. 
\begin{proof}[Proof of Theorem \ref{ThmTightness}]
Suppose we have an initially infected set $A_0$ of size $l$ that is percolating. Let $U$ and $W$ be the parts of $H$ and $U_t=A_t \cap U$ and $W_t=A_t \cap W$ for $t\in \mathbb{N}$. Note that by the degree condition $\vert U_0 \vert, \vert W_0 \vert \geq  l-r+f(l-r)$ otherwise there will be one side where we will never be able to infect a vertex. We assume without loss of generality that $U$ is the side where we infect $f(l-r)$ vertices first, i.e. suppose there exists a time step $t$ such that $\vert U_{t} \setminus U_0 \vert \geq f(l-r)$ while $\vert U_{t-1} \setminus U_0 \vert < f(l-r)$ and $\vert W_{t-1} \setminus W_0 \vert < f(l-r)$. Let $u_1, \dots u_{f(l-r)}$ be the first $f(l-r)$ newly infected vertices in $U$ that were infected in that order. Since $\vert U_0 \vert =l-\vert W_0 \vert$ we know that $u_1$ needs to have at least $r-\vert U_0\vert=\vert W_0 \vert -(l-r)$ infected neighbours in $W$ at the time it gets infected. Similarly $u_i$ needs to have at least $\vert W_0 \vert-(l-r)-i+1$ infected neighbours in $W_{t-1}$ at the time it gets infected. By Lemma \ref{Lemmatightness} and by counting the vertices in $W_{t-1}$ we get
\begin{align*}
\vert W_{t-1}\vert& \geq \vert N_{W_{t-1}} (u_1,\dots,u_{f(l-r)})\vert \\
&\geq \left( \sum_{i=1}^{f(l-r)} \vert W_0 \vert-(l-r)-i+1 \right)-\left(f(l-r)-1\right) \\
& = \vert W_0 \vert -(l-r) +\sum_{i=1}^{f(l-r)-2}i + (f(l-r)-1)(\vert W_0 \vert-(l-r)-f(l-r))\\
& > \vert W_0 \vert+ (f(l-r)-1) \left(\vert W_0 \vert-(l-r)-f(l-r)\right). 
\end{align*}
Case 1: $\vert W_0 \vert >l-r+f(l-r)$.\\
We get automatically a contradiction with the above inequality since $\vert W_{t-1} \setminus W_0 \vert \leq f(l-r)-1$.\\
\ \\
Case 2: $\vert W_0 \vert =l-r+f(l-r)$ and $\vert W_{t-1}\setminus W_0 \vert =0$.\\ 
This gives us also a contradiction.\\ 
\ \\
Case 3: $\vert W_0 \vert =l-r+f(l-r)$ and $\vert W_{t-1}\setminus W_0 \vert =1$.\\ 
Let $w_j$ be the $j-$th newly infected vertex in $W_{t-1}\setminus W_0$.
 Observe that the first newly infected vertex $w_1$ can have at most $\vert W_0\vert =l-r+f(l-r)$ infected neighbours in $W$ and therefore by the degree conditions of $H$ all of its $2r-l-f(l-r)$ neighbours in $U$ are infected already. Therefore the $\vert W_0 \vert-(l-r)-i+1$ infected neighbours we counted before for each $u_i$ had to be in $W_{t-1}\setminus \{w_1\}$ as $w_1$ cannot be infected before any of its neighbours in $U$. If $\vert W_{t-1}\setminus W_0 \vert =1$, we get therefore the following contradiction
\begin{align*}
\vert W_0 \vert =\vert W_{t-1} \setminus\{w_1\}\vert& \geq  \left( \sum_{i=1}^{f(l-r)} \vert W_0 \vert-(l-r)-i+1 \right)-\left(f(l-r)-1\right) >\vert W_0 \vert.
\end{align*}
Case 4: $\vert W_0 \vert =l-r+f(l-r)$ and $\vert W_{t-1}\setminus W_0 \vert =k \geq 2$.\\ Note that in the calculations before we lower bounded $ \vert N_{W_{t-1}}(u_i) \vert $ by $\vert W_0 \vert-(l-r)-i+1 $ but these were the minimum number of vertices that were infected before $u_i$ and we might have missed vertices in  $N_{W_{t-1}}(u_i)$ for example those neighbours in $W_{t-1}$ which were infected after $u_i$. Recall that $I_{U_t\setminus U_0}(w_j)$ are the infected neighbours of $w_j$ in $U_t\setminus U_0$ at the moment of its infection. We have for each $w_j$ at least not included $\vert I_{U_t\setminus U_0}(w_j)\vert$ adjacencies between $U_t\setminus U_0$ and $W_{t-1}$. We want to calculate a lower bound on $\sum_{j=1}^k \vert I_{U_t\setminus U_0}(w_j)\vert$ next to get a more precise bound on $ \sum_{i=1}^{f(l-r)}  \vert N_{W_{t-1}}(u_i) \vert $. We have by Lemma \ref{Lemmatightness}
\begin{align*}
\vert U_{0} \vert & \geq \left( \sum_{j=1}^k  \deg_{\, U_0}(w_j)\right)-(k-1)\\
& \geq \left( \sum_{j=1}^k  2r-l-f(l-r)-(j-1) - \vert I_{U_t\setminus U_0}(w_j)\vert
\right)-(k-1) 
\end{align*}
and using $\vert U_0 \vert=r-f(l-r)$ and $3r \geq 2l+f(l-r)+4$
\begin{align*}
\left(\sum_{j=1}^k \vert I_{U_t\setminus U_0}(w_j)\vert \right) & \geq \left( \sum_{j=3}^k  2r-l-f(l-r)-j+1\right)-k+4\\
& \geq (k-2)(l-r+4)-\frac{k(k-1)}{2}-k+5\\
& \geq  k
\end{align*}
where we use in the last step that $l-r\geq \frac{(f(l-r)-2)(f(l-r)-3)}{2}\geq \frac{(k-1)(k-2)}{2}$.
By using Lemma \ref{Lemmatightness} again and doing the same calculations as before, we get
\begin{align*}
\vert W_{t-1} \vert & \geq  \left( \sum_{j=1}^{f(l-r)} \vert W_0 \vert-(l-r)-j+1 \right)+k-\left(f(l-r)-1\right)\\
& > \vert W_0 \vert+k 
\end{align*}
which is a contradiction since $\vert W_{t-1} \setminus W_0 \vert =k$.
\end{proof}
The next corollary tells us that if $l$ is closer to $2r$, we can still find a graph $G$ with a slightly lower value $D(G)$ which does not have a percolating set of size $l$.
\begin{Cor}
Given $l\geq r$ such that $2r-1\geq l+f(l-r)$, let $n$ be sufficiently large. Let $H$ be a $(2r-l-f(l-r)-1)$-regular bipartite graph with $2n$ vertices and girth at least $2f(l-r)+2$. The graph $G$ obtained from $H$ by adding all edges inside each part of $H$ has $m(G,r) > l$.
\label{Corweak}
\end{Cor}
\begin{proof}
Note that the degree condition of $H$ means that we need to start in each part of $H$ with at least $l-r+f(l-r)+1$ infected vertices and we showed in Case $1$ of Theorem \ref{ThmTightness} above that in this case we cannot find a percolating set of size $l$ even if every vertex in $H$ had degree $2r-l-f(l-r)$.
\end{proof}
The following proposition says that the graph $G$ from Theorem \ref{ThmTightness} exists for sufficiently large $n$.
\begin{Proposition}
\label{PropTightness}
Let $l\geq r$. For $n$ sufficiently large and $2r\geq l+f(r-l)$ there always exists a $(2r-l-f(l-r))$-regular bipartite graph $H$ with $2n$ vertices and girth at least $2f(l-r)+2$. 
\end{Proposition}
This can be followed by a theorem of Erd\H{o}s and Sachs \cite{erdos1963regulare} and we give the proof in the Appendix. 
Note that so far we only gave a lower bound if $n$ is even. The next corollary gives us a lower bound also in the case when $n$ is odd.
\begin{Cor}
Given $l\geq r$ such that $3r \geq 2l +f(l-r)+4$, let $n$ be sufficiently large. Then there exists a graph $G$ on $n$ vertices such that $\delta(G)=\floor*{\frac{n}{2}}+2r-l-f(l-r)-1$ and $m(G,r)>l$. 
\end{Cor}
\begin{proof}
If $n$ is even, we take a graph $G$ with the properties as in Theorem \ref{ThmTightness}. If $n$ is odd, we take a graph $G^\prime$ on $n+1$ vertices with the properties from Theorem \ref{ThmTightness} and delete a vertex $v$ from one of its sides. We claim that the resulting graph $G$ has the desired properties. Note that the minimum degree of $G^\prime$ is $\frac{n+1}{2}+2r-l-f(l-r)-1$ and by deleting $v$ we can only decrease this minimum degree by one. Since $\frac{n+1}{2}-1=\floor*{\frac{n}{2}}$ we get the desired minimum degree for $G$. Note that $G$ has no percolating set of size $l$ otherwise infecting the corresponding vertices in $G^\prime$ would infect all of $G^\prime$.
\end{proof}
\subsection{Sufficient Ore-Type conditions for small percolating sets}
\label{chapter5enough}
We will show that if $G$ is a graph with $D(G)\geq n+ 4r-2l-2f(l-r)-1$, then $m(G,r)\leq l$. It is easy to prove that if we increase the degree in our bipartite subgraph $H$ from Theorem \ref{ThmTightness} by one that we can find a percolating set of size $l$, i.e. that we chose the degree of our regular bipartite graph in the tightness construction best possible. It is easy because the graph has a specific structure. In general, we do not know how our graph looks like as we only have the Ore-type condition. Therefore we want to investigate first how our graph can be structured and then use the structure to find a percolating set of size $l$. Similar as in Theorem \ref{ThmAsymptotic} we will show that if we start with a specific initially infected set of size $r$, either this set percolates which gives us a percolating set of size $r\leq l$, or otherwise we can infect at least roughly $\frac{n}{2}$ vertices and separate the graph into two parts $A$ and $A^c$ such that the bipartite graph between them is sparse. We will use that structural information to find a percolating set of size $l$.\\

Recall that for the graph $K_{r,s}$, by infecting the vertices on the side of order $r$ we can also infect the vertices on the side of order $s$. In fact, we will show that we can assume that the vertices of the $K_{r,s}$ have all high degree which will help us to infect roughly $\frac{n}{2}$ vertices.

Instead of looking at $D(G) \geq n+ 4r-2l-2f(l-r)-1$, we prove the following lemmas more generally for graphs with $D(G) \geq n+ 2r-m$ as it will be useful in Section \ref{resultsforbigk}.
\begin{Lemma}
Let $r,s \in \mathbb{N}$ and $m$ be an integer. For $n$ sufficiently large, any $n$-vertex graph $G$ with $D(G) \geq n+ 2r-m$ has a $K_{r,s}$ subgraph of vertices which have at least $\ceil*{\frac{n-m}{2}}+r$ neighbours in $G$ or a percolating set of size $r$.
\label{Lemma3b}
\end{Lemma}
\begin{proof}
Let $L$ be the set of vertices with degree less than $ \ceil*{\frac{n-m}{2}}+r$ and $M$ the set of vertices of degree at least $\ceil*{\frac{n-m}{2}}+r$. Note that by the Ore-type conditions $L$ needs to be a clique. If $\vert L\vert \leq r-1$, then each vertex in $M$ has degree at least $\ceil*{\frac{n-m}{2}}+1$ in $M$ so by Kovári, Sós and Turán \cite{kovari1954problem} it contains a $K_{r,s}$ for $n$ large enough.
 
Suppose $\vert L\vert \geq r$. Now infect $r$ vertices in $L$. Let $A$ be the closure of those $r$ vertices. Suppose $\vert A \vert <n$ otherwise we found a percolating set of size $r$. Since $L$ is a clique, we know $L \subset A$ and every element in $A^c$ has at most $r-1$ neighbours in $A$ and has degree at least $\ceil*{\frac{n-m}{2}}+r$. Deleting $A$ from the graph results in a graph that has minimum degree at least $\ceil*{\frac{n-m}{2}}+1$ and is of size at least $\ceil*{\frac{n-m}{2}}+2$ so contains a $K_{r,s}$ by Kovári, Sós and Turán \cite{kovari1954problem} if $n$ is large enough. 
\end{proof}
Now we will show that if we have a $K_{r,s}$ where each vertex has degree at least $\ceil*{\frac{n-m}{2}}+r$, we will actually infect many vertices.
\begin{Lemma} Let $m,r \in \mathbb{N}$.  If $G$ is a an $n$-vertex graph with $D(G) \geq n+ 2r-m$, then we can find an initially infected set of $r$ vertices which infects at least $\frac{n}{2}-o(n)$ vertices of $G$.
\label{Lemma4}
\end{Lemma}
\begin{proof} 
Note that by Lemma \ref{Lemma3b} we can assume that we have a subgraph isomorphic to $K_{r,s}$ of vertices of degree at least $\ceil*{\frac{n-m}{2}}+r$ in $G$. We infect the part with $r$ vertices of the $K_{r,s}$ which infects $r+s$ vertices of degree at least $\ceil*{\frac{n-m}{2}}+r$ and let $A$ be its closure. We get
\begin{align*}
(r-1) \vert A^c \vert & \geq \left\vert \{ \textrm{edges between }A^c \textrm{ and } A\}\right\vert \\
& \geq (r+s)\left(\ceil*{\frac{n-m}{2}}+r-(\vert A \vert -1)\right) \\
& =(r+s)\left(\ceil*{\frac{n-m}{2}}+r+1\right)-(r+s)\vert A \vert
\end{align*}
and therefore 
\begin{align*}
(r-1)\vert A^c \vert + (r+s) \vert A \vert=(r-1)n+(s+1) \vert A \vert \geq (r+s)\left(\ceil*{\frac{n-m}{2}}+r+1\right) .
\end{align*}
Using the last inequality we get that $\vert A \vert \geq \frac{s-r+2}{s+1} \left( \ceil*{\frac{n}{2}}+o(n) \right)$. But now for $n\to \infty$ the maximal $s$ such that we can find a $K_{r,s}$ grows and since $\frac{s-r+2}{s+1}$ converges to $1$ for $s \to \infty$ we proved the lemma.
\end{proof}
By the above lemma we can assume that our graph has a large closed set $A$ and this will help us to find a percolating set of size $l$. We want to examine the structure between $A$ and $A^c$ more first.
\begin{Lemma} Let $r,m \in \mathbb{N}$ and $G$ be a graph with $D(G) \geq n+ 2r-m$. If $A$ is a closed set and $n \neq \vert A \vert \geq r$, then we know that any vertex $x$ in $A$ has at least $r-m+3$ neighbours in $A^c$. Moreover, if $x$ is not connected to all vertices in $A^c$, we even have $\deg_{A^c}(x)  \geq  r-m+3 + \left\vert \{ \textrm{non-neighbours of } x \textrm{ in } A\}\right\vert$.
\label{Lemma3}
\end{Lemma}
\begin{proof}
Since $n \neq \vert A \vert \geq r$ for any $y \in A^c$ there exists $x \in A$ such that $y$ is not adjacent to $x$ because $y$ has at most $r-1$ neighbours in $A$. Then
\begin{align*}
n+2r-m & \leq \deg(x)+\deg(y)\\
& \leq \vert A \vert-1- \left\vert \{ \textrm{non-neighbours of } x \textrm{ in } A\}\right\vert + \deg_{A^c}(x)+\vert A^c \vert-1+\deg_A(y) \\
& \leq n-2- \left\vert \{ \textrm{non-neighbours of } x \textrm{ in } A\}\right\vert+\deg_{A^c}(x)+r-1\\
&=n+r-3+\deg_{A^c}(x)- \left\vert \{ \textrm{non-neighbours of } x \textrm{ in } A\}\right\vert
\end{align*} 
and therefore $\deg_{A^c}(x) - \left\vert \{ \textrm{non-neighbours of } x \textrm{ in } A\} \right\vert \geq  r-m+3$. Moreover $\vert A^c \vert \geq r-m+3$ and it follows that each vertex in $A$ is adjacent to at least $r-m+3$ vertices in $A^c$ as it is either connected to all in $A^c$ or otherwise it has a non-neighbour in $A^c$ and the above calculations apply. 
\end{proof}
Although we know that every vertex in $A^c$ has at most $r-1$ neighbours in $A$ there can be elements in $A$ which have many neighbours in $A^c$. So we know that the bipartite graph between $A$ and $A^c$ is sparse but must not be very regular. The following lemma helps us to say more about the bipartite graph between $A$ and $A^c$. Additionally, we say something about the structure within $A$ and within $A^c$.
\begin{Proposition} 
Given $m\in \mathbb{N}$, $r\geq 2$, let $n$ be sufficiently large. Let $G$ be a graph on $n$ vertices such that the sum of the degrees of any two non-adjacent vertices $x$ and $y$ is $D(G) = \deg(x)+\deg(y) \geq n+ 2r-m$. Then we either have a percolating set of size $l$ or a closed set $A$ of size $n/2-o(n) \leq \vert A \vert $ such that
\begin{itemize}
\item for $C=\{v \in A \colon\ \deg_{A^c}(v)\geq r\} $ we have $\vert C \vert \leq m+r-l-4$,
\item any $r$ infected vertices in $A^c$ infect all of $A^c \cup C$,
\item and any $r$ infected vertices in $A\setminus C$ infect all of $A$.
\end{itemize}
\label{Proposition5}
\end{Proposition}
\begin{proof}
Note that by Lemma \ref{Lemma4} we can assume that we have a closed set $A$ of size at least $\frac{n}{2} -o(n)$.  First we will prove that any $r$ infected vertices in $A^c$ infect all of $A^c$ and thus all of $A^c\cup C$. Observe that if $A^c$ has size at most $\ceil*{\frac{n-m}{2}}+2$, then it needs to be a clique as every element in $A^c$ can have degree at most $\vert A^c \vert+r-2$ and for two non adjacent vertices $x$ and $y$ in $A^c$ they would have at most $\deg(x)+\deg(y)\leq 2\vert A^c \vert+2r-6<n+2r-m$. Otherwise if $A^c$ is larger, we know that either a vertex $x \in A^c$ is connected to all vertices in $A^c$ or has a non-neighbour $y$ in $A^c$ but then it needs to have degree at least $\frac{n}{2}+o(n)$ otherwise $\deg(x)+\deg(y)\leq \deg(x)+(\frac{n}{2}+o(n))<n-o(n)$ which is a contradiction to the Ore-type condition. But then we get by Lemma \ref{Lemma1} that if $n$ is large enough, infecting $r$ vertices in $A^c$ infects all of $A^c$ and therefore all of $A^c \cup C$.\\
\ \\
We will now show that we can assume that $C$ is small. If $\vert C \vert \geq m+r-l-3$, we infect $r$ vertices in $A^c$ and $l-r$ in $A \setminus C$. Note that this infects all of $A^c$ and therefore also all vertices in $C$ which means we have at least $m-3$ infected vertices in $A$. Each not initially infected vertex  $x$ in $A\setminus C$ is connected to at least $m-3-\left\vert\{ \textrm{non-neighbours of } x \textrm{ in } A\}\right\vert$ infected vertices in $A$ and $r-m+3+\left\vert\{ \textrm{non-neighbours of } x \textrm{ in } A\}\right\vert$ infected vertices in $A^c$ by Lemma \ref{Lemma3} .Therefore it has at least $r$ infected vertices in $A^c \cup C \cup \{\textrm{initially infected vertices in } A \setminus C\}$ and gets infected. But this means the whole graph gets infected and we found a percolating set of size $r$. \\
\ \\
We can suppose $\vert C \vert < m+r-l-3$. Each element $x$ in $A\setminus C$ has $\deg_{A^c}(x) \leq r-1$ and by Lemma \ref{Lemma3} also $\deg_{A^c}(x) \geq r-m+3+\left\vert \{ \textrm{non-neighbours of } x \textrm{ in } A\}\right\vert$ so we know that $\left\vert \{ \textrm{non-neighbours of } x \textrm{ in } A\} \right\vert \leq m-2$. But that means by Lemma \ref{Lemma1} that any $r$ vertices in $A \setminus C$ infect all of $A \setminus C$ for large enough $n$. \\
The only property that we are missing is that any $r$ infected vertices in $A \setminus C$ infect not only $A \setminus C$ but all of $A$. This is not automatically given. Now instead of considering $A$, we will consider $\tilde{A}=\langle A \setminus C\rangle \subset A$ and note that we deleted at most $o(n)$ vertices from $A$. Note that by the same reasoning as above we can assume that $\tilde{C}=\{v \in \tilde{A}  \colon\ \deg_{\tilde{A}^c}(v)\geq r\} $ has $\vert \tilde{C} \vert \leq m+r-l-3$ and any $r$ vertices in $\tilde{A}^c$ infect all of $\tilde{A}^c \cup \tilde{C}$. Note that $\tilde{A} \setminus \tilde{C} \subset A \setminus C$ and therefore any $r$ infected vertices in $\tilde{A} \setminus \tilde{C}$ infect all of $A\setminus C$ and since $\tilde{A}=\langle A \setminus C\rangle$ in the end all of $\tilde{A}$. Therefore we choose $A=\tilde{A}$.
\end{proof}
We will use the structure we encountered above and a case distinction on the bipartite graph $H$ between $A$ and $A^c$ to find a percolating set of size $l$ in a graph $G$ with $D(G)\geq n+ 4r-2l-2f(l-r)-1$. For this, we prove two lemmas that tell us that if we have some special structure, we can find a set of size $l$ which infects many vertices and which will be a percolating set in the end. $U$ and $W$ in the next lemma will be later roughly $A$ and $A^c$. 
\begin{Lemma}
Let $G$ be a graph with $n$ vertices $U\sqcup W$ and $0\leq j\leq f(l-r)-2$. Let all vertices in $U$ have at least $j$ neighbours in $W$. Let $U$ grow with $n$ and each $u\in U$ have at most $O(1)$ non-neighbours in $U$. Then for sufficiently large $n$ we can find an initially infected set of at most $l$ vertices in $U\sqcup W$ which infects at least all of $U$ and $j+l-r$ vertices in $W$ or if $\vert W \vert <j+l-r$, then all of $W$.  
\label{Lemmaf(k)-2}
\end{Lemma}
\begin{proof}
We take a clique of $j$ vertices $u_1, \dots, u_{j}$ in $U$ which exists for $n$ large enough since any vertex in $U$ has at most $O(1)$ non-neighbours in $U$. If $\vert W \vert <l-r+j$, let $W_0=W$.  Otherwise let $W_0$ consist of $j-i+1$ neighbours of each $u_i$, which exist by assumption, and add other arbitrary vertices from $W$ to $W_0$ until 
\begin{align*}
\vert W_0 \vert=j+l-r \geq j+\frac{(f(l-r)-2) (f(l-r)-3)}{2} \geq j+\frac{j (j-1)}{2}= \sum_{i=1}^{j} i.
\end{align*}
For $U_0$ we take $r-j$ common neighbours of $u_1, \dots, u_{j}$ in $U$ as in Figure \ref{Bild3}. We initially infect $U_0$ and $W_0$. Note that each $u_i$ has $r$ neighbours in $U_0 \cup W_0 \cup \{u_1,\dots,u_{i-1}\}$ and therefore in the $r$-neighbour bootstrap process we can infect $u_1$ and then $u_2$ and so on. Since there are $r$ infected vertices $U_0 \cup \{u_1, \dots, u_{j}\}$  in $U$ and each vertex in $U$ has $O(1)$ non-neighbours in $U$ by Lemma \ref{Lemma1} we infect all of $U$ for $n$ large enough. \
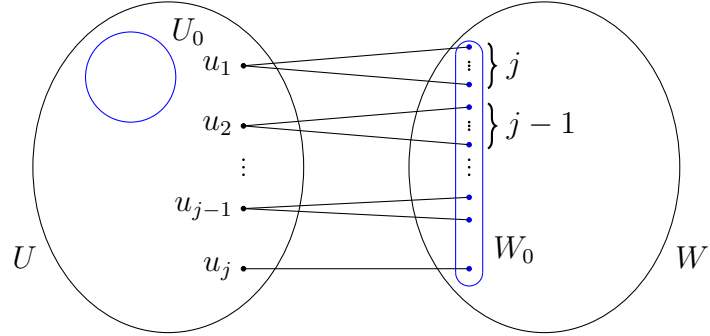
\begin{figure}[h]
\centering
\begin{tikzpicture}
    \draw (-1,1.5) ellipse (1.8cm and 2.2cm) node[left,xshift=-1.6cm,yshift=-1.2cm]{$U$};
    \draw (4,1.5) ellipse (1.8cm and 2.2cm) node[right,xshift=1.6cm,yshift=-1.2cm]{$W$};
     \draw (-1.5,2.7) ellipse (0.6cm and 0.6cm) [blue] node[right,xshift=0.4cm,yshift=0.6cm,black]{$U_0$};  
    
   \foreach \i in {2.8,2}
   { \draw [fill,blue] (3,\i+0.3) circle [radius=0.03];
    \draw [fill,blue] (3,\i+0.8-1) circle [radius=0.03];
    \draw [fill](3,\i +0.1) circle [radius=0.01];
    \draw [fill](3,\i +0.05) circle [radius=0.01];
    \draw [fill](3,\i ) circle [radius=0.01];
    \draw (0,\i+0.05) -- (3,\i+0.3);
    \draw (0,\i+0.05) -- (3,\i+0.8-1);
    }
    \draw [fill] (0,2.8+0.05) circle [radius=0.03] node[left]{$u_1$};
    \draw [fill] (0,2+0.05) circle [radius=0.03] node[left]{$u_2$};

    \draw [fill,blue] (3,0.15) circle [radius=0.03];
   \draw [fill,blue] (3,1.1) circle [radius=0.03];
   \draw [fill,blue] (3,0.8) circle [radius=0.03];
   \draw [fill] (0,0.9+0.05) circle [radius=0.03] node[left]{$u_{j-1}$};
        \draw (0,0.9+0.05) -- (3,1.1);
        \draw (0,0.9+0.05) -- (3,0.8);
   \foreach \i in {0.1}
   {\draw [fill] (0,\i+0.05) circle [radius=0.03] node[left]{$u_{j}$};
     \draw (0,\i+0.05) -- (3,\i+0.05);
    }    
    
   \foreach \i in {1.4,1.5,1.6}
   { \draw [fill] (0,\i) circle [radius=0.01];
   \draw [fill](3,\i) circle [radius=0.01];
   }

     \draw[thick,black,decorate,decoration={brace,amplitude=0.1cm}] (3.25,3.15) -- (3.25,2.55) node[midway, right,xshift=0.1cm]{$j$};
         \draw[thick,black,decorate,decoration={brace,amplitude=0.105cm}] (3.25,2.35) -- (3.25,1.75) node[midway, right,xshift=0.1cm]{$j-1$};

  \draw [blue, domain=0:180] plot ({0.18*cos(\x)+3}, {0.18*sin(\x)+3});
  \draw [blue] (3-0.18,3) -- (3-0.18, 0.1);
   \draw [blue, domain=180:360] plot ({0.18*cos(\x)+3}, {0.18*sin(\x)+0.1});
   \draw [blue] (3+0.18,3) -- (3+0.18, 0.1);
    \node at (3.6,0.4){$W_0$};
\end{tikzpicture}
\caption{We choose $U_0\subset U$ to be a set of $r-j$ neighbours of $u_1,\dots, u_j$  and $W_0\subset W$ a set of size $l-r+j$ such that every $u_i$ has $j-i+1$ neighbours in $U_0$ for each $1\leq i \leq j$. If $U$ is dense, infecting $U_0$ and $W_0$ infects all of $U$.}
\label{Bild3}
\end{figure}
\end{proof}
The above lemma does not hold if we wanted to take $j=f(l-r)-1$ as we would need to infect more than $l$ vertices in the beginning as $r-(f(l-r)-1)+\sum_{i=1}^{f(l-r)-1}i > l$. Under some additional conditions on $W$ we can still find a subgraph which helps us to infect all of $U$ and $l-r+f(l-r)-1$ vertices in $W$ as described in the following lemma.
\begin{Lemma}
Let $G$ be a graph with $n$ vertices $U\sqcup W$. Let all vertices in $U$ have at least $f(l-r)-1$ neighbours in $W$ and all vertices in $W$ at least $f(l-r)$ neighbours in $U$. Let $U$ grow with $n$ and let there be at most one vertex $v$ in $U$ with at least $2$ non-neighbours in $U$. Then for sufficiently large $n$ we can find an initially infected set of $l$ vertices in $U\sqcup W$ which infects at least all of $U \setminus \{v\}$ and $l-r+f(l-r)-1$ vertices in $W$. 
\label{Lemmaonenonneighbour}
\end{Lemma}
\begin{proof}
We choose a vertex $u_1$ in $U$ which is a non-neighbour of $v$ if $v$ exists. We can do the same as in Figure \ref{Bild3d}. We choose a neighbour $w_1$ of $u_1$ and we choose a neighbour $u_2$ of $w_1$ such that it is a neighbour of $u_1$ which is possible since $u_1$ has at most one non-neighbour and $w_1$ has $f(l-r)-1$ neighbours in $U$. We continue like this adding $u_{i+1}$ in such a way that we are avoiding non-neighbours of $u_1, \dots, u_i$ which is possible since we need to avoid at most $f(l-r)-2$ vertices. If we need to pick $u_i=u_j$ for some $j<i$ and create a cycle, then we pick $u_{i+1}$ as an arbitrary neighbour of $u_1,\dots,u_i$ and start from the beginning. Note that we get some cycles plus at most one path in the bipartite graph between $U$ and $W$. We add to $W_0$ the vertices $w_1, \dots, w_{f(l-r)-2}$ plus for every $u_i$ which is either in a cycle or not an end vertex of a path we add $f(l-r)-i-2$ other neighbours of $u_i$ in $W$ to $W_0$ and otherwise if $u_i$ is an end vertex of a path, we add $f(l-r)-i-1$ neighbours of $u_i$ to $W_0$. We possibly add more arbitrary vertices from $W$ to $W_0$ to get exactly 
\begin{align*}
l-r+f(l-r)-1 &\geq \frac{(f(l-r)-2)(f(l-r)-3)}{2}+f(l-r)-1 \\
&= \left(\sum_{i=-1}^{f(l-r)-3} i \right)+2+f(l-r)-2
\end{align*}
vertices in $W_0$. Let $U_0$ be $r-f(l-r)+1$ shared neighbours of $u_1, \dots, u_{f(l-r)-1}$. Then infecting $W_0$ and $U_0$ infects all of $U\setminus \{v\}$ and we proved the lemma.

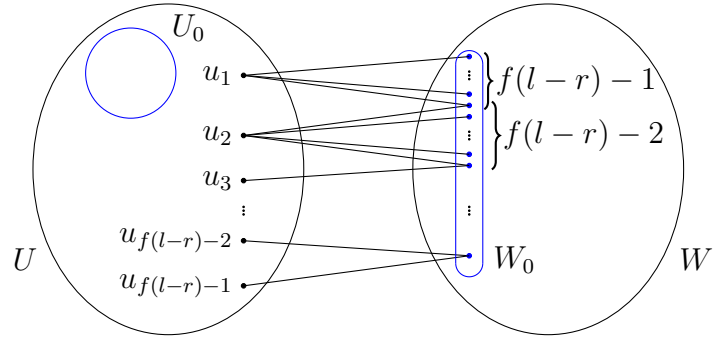
\begin{figure}[h]
\centering
\begin{tikzpicture}
    \draw (-1,1.6) ellipse (1.8cm and 2.2cm) node[left,xshift=-1.6cm,yshift=-1.2cm]{$U$};
    \draw (4.05,1.6) ellipse (1.8cm and 2.2cm) node[right,xshift=1.6cm,yshift=-1.2cm]{$W$};
    \draw (-1.5,2.88) ellipse (0.6cm and 0.6cm) [blue] node[right,xshift=0.4cm,yshift=0.6cm,black]{$U_0$};  
    
   \foreach \i in {2.8,2}
   { \draw [fill,blue] (3,\i+0.3) circle [radius=0.03];
    \draw [fill,blue] (3,\i+0.8-1) circle [radius=0.03];
    \draw [fill](3,\i +0.1) circle [radius=0.01];
    \draw [fill](3,\i +0.05) circle [radius=0.01];
    \draw [fill](3,\i ) circle [radius=0.01];
    \draw (0,\i+0.05) -- (3,\i+0.3);
    \draw (0,\i+0.05) -- (3,\i+0.8-1);
    }

    \draw [fill] (0,2.8+0.05) circle [radius=0.03] node[left]{$u_1$};    
    \draw [fill] (0,2+0.05) circle [radius=0.03] node[left]{$u_2$};
    \draw [fill] (0,0.05) circle [radius=0.03] node[left] {$u_{f(l-r)-1}$};
        \draw [fill] (0,1.4+0.05) circle [radius=0.03] node[left]{$u_3$};
    \draw [fill] (0,0.6+0.05) circle [radius=0.03] node[left] {$u_{f(l-r)-2}$};
    \draw (0,1.4+0.05) -- (3,1.65);
    \draw (0,0.6+0.05) -- (3,0+0.45);
      \foreach \i in {2,0}
   { \draw [fill,blue] (3,\i+0.45) circle [radius=0.03];
     \draw (0,\i+0.05) -- (3,\i+0.45);
    }
     \draw (0,2.8+0.05) -- (3,2+0.45);
     \draw [fill,blue] (3,1.2+0.45) circle [radius=0.03];
     \draw (0,2+0.05) -- (3,1.2+0.45);
         \draw [fill](3,1 +0.1) circle [radius=0.01];
    \draw [fill](3,1 +0.05) circle [radius=0.01];
    \draw [fill](3,1) circle [radius=0.01];
             \draw [fill](0,1 +0.1) circle [radius=0.01];
    \draw [fill](0,1 +0.05) circle [radius=0.01];
    \draw [fill](0,1) circle [radius=0.01];
     \draw[thick,black,decorate,decoration={brace,amplitude=0.1cm}] (3.2,3.15) -- (3.2,2.4) node[midway, right]{$f(l-r)-1$};
         \draw[thick,black,decorate,decoration={brace,amplitude=0.105cm}] (3.3,2.5) -- (3.3,1.6) node[midway, right]{$f(l-r)-2$};
  \draw [blue, domain=0:180] plot ({0.18*cos(\x)+3}, {0.18*sin(\x)+3});
  \draw [blue] (3-0.18,3) -- (3-0.18, 0.35);
   \draw [blue, domain=180:360] plot ({0.18*cos(\x)+3}, {0.18*sin(\x)+0.35});
   \draw [blue] (3+0.18,3) -- (3+0.18, 0.35);
    \node at (3.6,0.4){$W_0$};
\end{tikzpicture}
\caption{Let $U_0 \subset U$ be a set of $r-f(l-r)+1$ shared neighbours of $u_1, \dots, u_{f(l-r)-1}$ and $W_0\subset W$ be as set of size $l-r+f(l-r)-1$ which contains neighbours of $u_1, \dots, u_{f(l-r)-1}$ as indicated. If $U$ is dense, infecting $U_0$ and $W_0$ infects all of $U$.}
\label{Bild3d}
\end{figure}
\end{proof}
We are now able to prove our main result of this section. As before we will separate our vertex set into three sets $V(G)=(A\setminus C) \cup C \cup A^c$ and use the structure between those sets to find a percolating set of size $l$.

\begingroup
\def\theTheorems{\ref{mainthm}}
\begin{Theorems}
Let $l\geq r$ and $2r \geq l+2f(l-r)-1$. For sufficiently large $n$, any $n$-vertex graph $G$ with $D(G) \geq n+ 4r-2l-2f(l-r)-1$ has a percolating set of size $l$. 
\end{Theorems}
\addtocounter{Theorems}{-1}
\endgroup

\begin{proof}
By Proposition \ref{Proposition5} we can assume we have a closed set $A$ such that any $r$ infected vertices in $A^c$ infect all of $A^c \cup C$ and any $r$ infected vertices in $A\setminus C$ infect all of $A$ and $\vert C \vert \leq l-r+2f(l-r)-3$. Moreover, $\vert A \vert \geq \frac{n}{2}-o(n)$.\\
\ \\
Let $i$ be defined such that $\vert A \vert+2r-l-2f(l-r)+i=\min_{v \in A \setminus C}(\deg(v))$ and $i_c$ such that $\vert A^c \vert+2r-l-2f(l-r)+i_c=\min_{v \in A \setminus C}(\deg(v))$. Note that by the Ore-type condition all except at most $r-1$ vertices in $A^c$ have degree at least $\vert A^c \vert +2r-l-i-1$ and all except at most $r-1$ vertices in $A$ have degree at least $\vert A \vert +2r-l-i_c-1$. If we can infect $l-r+\max\{i,0\}$ vertices in $A^c$ and all of $A$, we will infect almost every vertex $v$ in $A^c$ which has at least $2r-l-i + \left\vert \{ \textrm{non-neighbours of } v \textrm{ in } A^c\}\right\vert$ neighbours in $A$ and then we will infect all of $A^c$. Otherwise if we can infect $l-r+\max\{i_c-\vert C \vert,0\}$ vertices in $A\setminus C$ and all of $A^c$ this will infect also $C$ and therefore we have $l-r+i_c$ infected vertices in $A$ which infects all of $A$ as almost every vertex in $A\setminus C$ has at least $2r-l-i_c+\left\vert \{\textrm{non-neighbours of } v \textrm{ in } A\}\right\vert$ neighbours in $A^c$. \\
\ \\
Case 1.1: $i \leq f(l-r)-2$. \\
We will infect every vertex in $A\setminus C$ which will infect all vertices in $A$. Additionally we infect $l-r+\max\{i,0\}$ vertices in $A^c$. This follows from the following claim.\\
Claim: We can verify the conditions of Lemma \ref{Lemmaf(k)-2} with $U= A \setminus C$ and $W=A^c$ and $j=\max\{i,0\}$. \\
Any vertex in $A\setminus C$ has at most $l-r+2f(l-r)-i-2$ non-neighbours in $A\setminus C$ as it has degree at least $\vert A \vert+2r-l-2f(l-r)+i$ and at most $r-1$ neighbours in $A^c$. $A$ grows with $n$ by assumption. Every vertex in $U= A \setminus C$ has at least $\max\{2r-l-2f(l-r)+i+1,0\}\geq j$ neighbours in $A^c$. \\
\ \\
Case 1.2: $i_c\leq f(l-r)-2$ or $i_c=f(l-r)-1$ and $\vert C \vert \geq 1$ or $i_c=f(l-r)$ and $\vert C \vert \geq 2$.\\
Claim: We can verify the conditions of Lemma \ref{Lemmaf(k)-2} with $U= A^c$ and $W=A\setminus C$ and $j=\max\{i_c-\vert C \vert,0\}$.\\
 Note that each vertex in $A^c$ has at least $\max\{i_c-\vert C \vert,0\}$ neighbours in $A\setminus C$. Any vertex in $A^c$ has at most $l-r+2f(l-r)-i_c-2$ non-neighbours in $A^c$ as it has degree at least $\vert A^c \vert+2r-l-2f(l-r)+i_c$ and at most $r-1$ neighbours in $A$. The only technical difference is that we do not have a lower bound on $\vert A^c \vert$ but note that almost every element in $A$ must have at least $2r-l-i_c\geq f(l-r)-1$ neighbours in $A^c$ so we know that $A^c$ has to grow with $n$ as every vertex in $A^c$ can have at most $r-1$ neighbours in $A$. \\
\ \\
For the remainining cases we assume that $i \geq f(l-r)-1$ and $i_c \geq f(l-r)-1$ and we assume that if  $i_c=f(l-r)-1$ then $\vert C \vert =0$ and if $i_c=f(l-r)$ then $\vert C \vert \leq 1$.
Note that since $D(G)\geq n+4r-2l-2f(l-r)-1$ we have either that almost all vertices in $A \setminus C$ must have degree at least $\vert A \vert+2r-l-f(l-r)$ or that almost all vertices in $A^c$ must have degree at least $\vert A^c \vert+2r-l-f(l-r)$. Note that if $i_c=f(l-r)-1$ and $\vert C \vert=0$, we can exchange the roles of $A^c$ and $A$ and therefore we may assume that almost all vertices in $A^c$ have degree at least $\vert A^c \vert+2r-l-f(l-r)$.\\
\ \\
Case 2: $i_c\geq f(l-r)+1$. \\
Claim: We can apply Lemma \ref{Lemmaf(k)-2} with $U=A\setminus C$ and $W=A^c$ and $j=f(l-r)-2$.  \\
Every vertex in $A$ has at least $f(l-r)-2$ neighbours in $A^c$ and since every vertex in $A\setminus C$ has at most $r-1$ neighbours in $A^c$, also every vertex in $A\setminus C$ has at most $O(1)$ non-neighbours in $A$. Moreover, $A$ grows with $n$ by assumption. \\
\ \\
Case 3.1: $\vert C \vert=0$ and there exist at least two vertices $w_1,w_2$ in $A$ which have at least $2r-l-f(l-r)+2$ neighbours in $A^c$.\\
Claim: We can verify the conditions of Lemma \ref{Lemmaf(k)-2} with $j=f(l-r)-2$ and $U=A^c$ and $W=A$.\\
Note that each vertex in $A^c$ has at least $2r-l-f(l-r)\geq f(l-r)-1$ neighbours in $A$. The other conditions hold as explained before. The only difference to Lemma \ref{Lemmaf(k)-2} is that we choose $u_1, \dots, u_{f(l-r)-2} \in A^c$ such that they are only connected to neighbours of $w_1$ or $w_2$ in $A$ as in Figure \ref{Bild3b} (and form a clique). This is possible for $n$ sufficiently large since any vertex in $A$ has at most $r-1$ neighbours in $A^c$ and $w_1$ and $w_2$ have at most $l-r+f(l-r)-1$ non-neighbours in $A$ since they have at most $r-1$ neighbours in $A$ and degree at least $\vert A \vert+2r-l-f(l-r)-1$. Therefore we have only $O(1)$ vertices in $A$ which we need to avoid and hence only $O(1)$ vertices in $A^c$ which we can not choose for $u_1, \dots, u_{f(l-r)-2}$. As in Lemma \ref{Lemmaf(k)-2} we can infect all vertices in $A^c$ and $l-r+f(l-r)-2$ neighbours of $w_1$ and $w_2$ in $A$. But then $w_1$ and $w_2$ get infected also and this spreads the infection to all of the graph as we have $l-r+f(l-r)$ infected vertices in $A$. \\
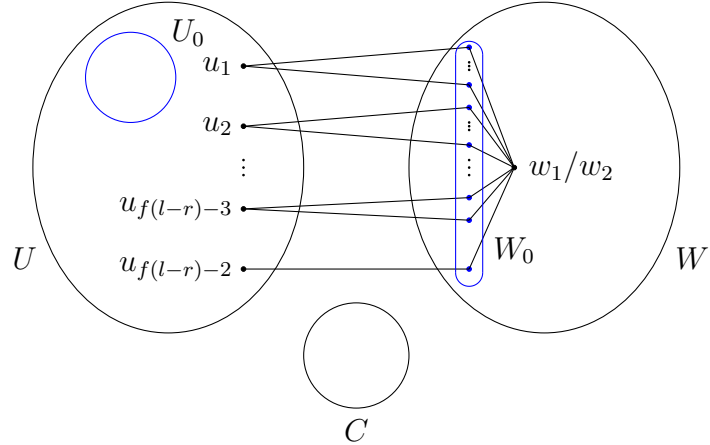
\begin{figure}[h]
\centering
\begin{tikzpicture}
    \draw (-1,1.5) ellipse (1.8cm and 2.2cm) node[left,xshift=-1.6cm,yshift=-1.2cm]{$U$};
    \draw (4,1.5) ellipse (1.8cm and 2.2cm) node[right,xshift=1.6cm,yshift=-1.2cm]{$W$};
    \draw (1.5,-1) circle (0.7cm) node[below,yshift=-0.7cm]{$C$};
    \draw (-1.5,2.7) ellipse (0.6cm and 0.6cm) [blue] node[right,xshift=0.4cm,yshift=0.6cm,black]{$U_0$};  
    
   \foreach \i in {2.8,2}
   { \draw [fill,blue] (3,\i+0.3) circle [radius=0.03];
    \draw [fill,blue] (3,\i+0.8-1) circle [radius=0.03];
    \draw (3.6,1.5) -- (3,\i+0.3);
    \draw (3.6,1.5) -- (3,\i+0.8-1);
    \draw [fill](3,\i +0.1) circle [radius=0.01];
    \draw [fill](3,\i +0.05) circle [radius=0.01];
    \draw [fill](3,\i ) circle [radius=0.01];
    \draw (0,\i+0.05) -- (3,\i+0.3);
    \draw (0,\i+0.05) -- (3,\i+0.8-1);
    }
        \draw [fill] (0,2.8+0.05) circle [radius=0.03] node[left]{$u_1$};
    \draw [fill] (0,2+0.05) circle [radius=0.03] node[left]{$u_2$};
    
    \draw [fill,blue] (3,0.15) circle [radius=0.03];
        \draw (3.6,1.5) -- (3,0.15);
   \draw [fill,blue] (3,1.1) circle [radius=0.03];
       \draw (3.6,1.5) -- (3,1.1);
   \draw [fill,blue] (3,0.8) circle [radius=0.03];
       \draw (3.6,1.5) -- (3,0.8);
   \draw [fill] (0,0.9+0.05) circle [radius=0.03] node[left]{$u_{f(l-r)-3}$};
        \draw (0,0.9+0.05) -- (3,1.1);
        \draw (0,0.9+0.05) -- (3,0.8);
   \foreach \i in {0.1}
   {\draw [fill] (0,\i+0.05) circle [radius=0.03] node[left]{$u_{f(l-r)-2}$};
     \draw (0,\i+0.05) -- (3,\i+0.05);
    }    
	\draw [fill] (3.6,1.5) circle [radius=0.03] node[right,xshift=0.05cm]{$w_1 / w_2$};	
	
	  \foreach \i in {1.4,1.5,1.6}
   { \draw [fill] (0,\i) circle [radius=0.01];
   \draw [fill](3,\i) circle [radius=0.01];
   }

  \draw [blue, domain=0:180] plot ({0.18*cos(\x)+3}, {0.18*sin(\x)+3});
  \draw [blue] (3-0.18,3) -- (3-0.18, 0.1);
   \draw [blue, domain=180:360] plot ({0.18*cos(\x)+3}, {0.18*sin(\x)+0.1});
   \draw [blue] (3+0.18,3) -- (3+0.18, 0.1);
    \node at (3.6,0.4){$W_0$};

\end{tikzpicture}
\caption{Let $U_0 \subset U$ be a set of $r-f(l-r)+2$ shared neighbours of $u_1, \dots, u_{f(l-r)-2}$ and $W_0\subset W$ be as set of size $l-r+f(l-r)-2$ such that each $u_i$ has at least $f(l-r)-1-i$ neighbours in $W_0$ for each $1\leq i \leq f(l-r)-2$. If $U$ is dense, infecting $U_0$ and $W_0$ infects all of $U$. If $w_1$ and $w_2$ have both at least $2r-l-f(l-r)+2$ neighbours in $U$, they will get infected also.}
 \label{Bild3b}
\end{figure}
\ \\
Case 3.2: $i_c  \geq f(l-r) $, $\vert C \vert=0$ and there exists at most one vertex $v$ in $A$ with at least $2r-l-f(l-r)+2$ neighbours in $A^c$.\\
Claim: We can apply Lemma \ref{Lemmaonenonneighbour} with $U=A$ and $W=A^c$. \\
Since all vertices in $A\setminus \{v\}$ have at most $2r-l-f(l-r)+1$ neighbours in $A^c$ they need to be connected to all other vertices in $A$ except at most one. \\
\ \\
Case 3.3:  $i_c=f(l-r)-1$, $\vert C\vert=0$ and there exists at most one vertex in $A$ with at least $2r-l-f(l-r)+2$ neighbours in $A^c$.\\
We cam assume $i=f(l-r)-1$, otherwise we can exchange the roles of $A$ and $A^c$ and are in Case 3.2.\\
Claim: We can infect $l-r+f(l-r)-1$ vertices in $A^c$ and all of $A$.\\
We take a vertex $u\in A^c$ which has only neighbours to those vertices in $A$ which have $2r-l-f(l-r)+1$ neighbours in $A^c$ and are connected to all vertices in $A$. Note that this is possible since there are at most $O(1)$ vertices in $A$ that have degree $\vert A \vert+2r-l+f(l-r)-1$ and almost all of those with degree $\vert A \vert+2r-l+f(l-r)$ have $2r-l-f(l-r)+1$ neighbours in $A^c$ and are hence connected to all vertices in $A$. We take $f(l-r)-1$ neighbours $w_1, \dots, w_{f(l-r)-1}$ of $u$ in $A$. Now we infect $u$, $f(l-r)-1-i$ other neighbours of $w_i$ in $A^c$ and more to get exactly $l-r+f(l-r)-1$ infected vertices in $A^c$ and arbitrary $r-f(l-r)+1$ vertices in $A$ which are not one of $w_1, \dots,  w_{f(l-r)-1}$. This infects $w_1, \dots,  w_{f(l-r)-1}$ and therefore all of $A$. \\
 \ \\
Case 4.1: $i_c=f(l-r)$, $\vert C \vert=1$ and there exists at least one vertex $v$ in $A^c$ with at least $2r-l-f(l-r)+2$ neighbours in $A$.\\
Claim: We can verify the conditions of Lemma \ref{Lemmaf(k)-2} with $j=f(l-r)-2$ and $U=A\setminus C$ and $W=A^c$ and choose $u_1,\dots,u_{f(l-r)-2}$  such that they are only connected to neighbours of $v$ in $A^c$. This is possible since $v$ has only $O(1)$ non-neighbours in $A^c$ and each vertex in $A^c$ has at most $r-1$ neighbours in $A \setminus C$. This infects all of $A$ and $l-r+f(l-r)-2$ neighbours of $v$ in $A^c$ and therefore also $v$.\\
\ \\
Case 4.2: $i_c=f(l-r)$, $\vert C\vert=1$ and all vertices in $A^c$ have at most $2r-l-f(l-r)+1$ neighbours in $A$, i.e. $A^c$ is a clique.\\
Claim: We can do the same strategy as in Figure \ref{Bild3d} with $U=A^c$ and $W=A\setminus C$.\\ Note that each vertex in $A^c$ has at least $f(l-r)-1$ neighbours in $A\setminus C$ and every vertex in $A\setminus C$ at least $2$ neighbours in $A^c$ so we can get a collection of cycles and at most one path by adding for each $u_i$ a neighbour $w_i$ to $W_0$ and choosing $u_{i+1}$ as a neighbour of $w_i$. If we need to pick $u_{i+1}=u_j$ for some $j<i$ and create a cycle, then we pick $u_{i+1}$ arbitrarily from $U\setminus \{u_1,\dots,u_i\}$. Note that $u_1,\dots, u_{f(l-r)-1}$ form a clique automatically. We add to $W_0$ at least $f(l-r)-i$ neighbours of $u_i$ and possibly some more to get exactly $l-r+f(l-r)-1$ infected vertices in $W$ and let $U_0$ consist of $r-f(l-r)-1$ arbitrary vertices from $U \setminus \{u_1,\dots,u_{f(l-r)-1}\}$. We can infect all of $A^c$ and $l-r+f(l-r)-1$ vertices in $A \setminus C$ and also the vertex in $C$.

\end{proof}

\subsection{Results for some larger $l$}
\label{resultsforbigk}
We have shown that if $3r \geq 2l +f(l-r)+4$, then $D_0(n,r,l)=n+ 4r-2l-2f(l-r)-1$ and if $2l-2r +f(l-r)+3 \geq r \geq l-r+2f(l-r)-1$, then $D_0(n,r,l)=n+ 2r-2k-2f(k)-1-j(r,k)$ where $j(r,k) \in \{0,1,2\}$ when $n$ satisfies some divisibility conditions. So far we do not know what happens for larger values of $l$. In the following we will determine $D_0(n,r,l)$ in the case when $l-r+ 2 f(l-r)-2 \geq r \geq l-r+ 3$.
\begin{Theorems}
Let $l-r+ 2 f(l-r)-2 \geq r \geq l-r+2$. For sufficiently large $n$, any $n$-vertex graph $G$ with $D(G)\geq n+2r-l-2$ has a percolating set of size $l$. 
\label{bigkbound}
\end{Theorems}
\begin{proof}
We know by Proposition $\ref{Proposition5}$ that we have a closed set $A$ of size $n/2-o(n)$ such that for $C=\{v \in A \colon\ \deg_{A^c}(v)\geq r\} $ we have $\vert C \vert \leq r-2$ and any $r$ infected vertices in $A\setminus C$ infect all of $C$ and any $r$ infected vertices in $A^c$ infect all of $A^c \cup C$.\\
\ \\
Let $i$ be defined such that $\vert A \vert+i-1=\min_{v \in A \setminus C}(\deg(v))$ and let $i_c$ be defined such that $\vert A^c \vert+i_c-1=\min_{v \in A^c}(\deg(v))$. If we can infect $l-r+\max\{i,0\}$ vertices in $A^c$ and all of $A$, we will infect almost every vertex $v$ in $A^c$ since almost every vertex has $\deg_A(v) \geq 2r-l-i + \left\vert \{ \textrm{non-neighbours of } v \textrm{ in } A^c\}\right\vert$. This infects all of $A^c$. Otherwise if we can infect $l-r+\max\{i_c-\vert C \vert,0\}$ vertices in $A\setminus C$ and all of $A^c$ this will infect also $C$ and therefore we have $l-r+i_c$ infected vertices in $A$ which infects all of $A$ as almost every vertex in $A\setminus C$ has at least $2r-l-i_c+\left\vert \{\textrm{non-neighbours of } v \textrm{ in } A\}\right\vert$ neighbours in $A^c$. \\
\ \\
If $i\leq f(l-r)-2$ we apply Lemma \ref{Lemmaf(k)-2} with $j=\max\{i,0\}$ and $U=A\setminus C$ and $W=A^c$ and this infects all of $A$ and $l-r+j$ vertices in $A^c$. \\
If $i_c \leq f(l-r)-2$, we apply Lemma  \ref{Lemmaf(k)-2} with $U=A^c$ and $W=A\setminus C$ and $j=\max\{i-\vert C \vert,0\}$. We do not have a lower bound on $\vert A^c \vert$ but almost every vertex in $A$ needs to have at least $2r-l-i\geq f(l-r)$ neighbours in $A^c$ and therefore $A^c$ needs to grow with $n$.\\
\ \\
The only case that is left is when  $i,i_c \geq f(l-r)-1 \geq 2r-l-f(l-r)+1$, i.e. all vertices in $A$ have degree at least $\vert A \vert+f(l-r)-2$ and all vertices in $A^c$ have degree at least $\vert A^c \vert+f(l-r)-2$. Note that as before, both $\vert A\vert$ and $\vert A^c \vert$ grow with $n$. If $\vert C \vert \geq 1$, we pick a vertex $v \in C$. We can infect by Lemma \ref{Lemmaf(k)-2} all of $A^c$ and $l-r+f(l-r)-2$ vertices in $A\setminus \{v\}$ as any vertex in $A^c$ has at least $f(l-r)-2$ neighbours in $A\setminus \{v\}$. But this infects also $v$ and we have $l-r+f(l-r)-1$ infected vertices in $A$. Therefore we assume $\vert C \vert=0$. Now we can use $A$ and $A^c$ interchangeably. Suppose we have two vertices $w_1,w_2$ in $A^c$ with at least $2r-l-f(l-r)+2$ neighbours in $A$. By Lemma \ref{Lemmaf(k)-2} we can infect all vertices in $A$ and at least $l-r+f(l-r)-2$ vertices in $A^c$ if we choose $U=A\setminus C$ and $W=A^c$ and $u_1,\dots,u_{f(l-r)-2}$ in such a way that they are only connected to neighbours of $w_1$ and $w_2$ in $A^c$ and therefore infect $w_1,w_2$ and all of $A^c$. If at most one vertex in $A^c$ has at least $2r-l-f(l-r)+2\leq f(l-r)$ neighbours in $A$, then all other vertices in $A^c$ need to be connected to all vertices in $A^c$, i.e. $A^c$ needs to be a clique and in particular $2r-l-f(l-r)+1= f(l-r)-1$. 
By the same argument we can assume that $A$ is a clique. We can easily find a substructure as in Figure \ref{Bild3d} of cycles and at most one path as described in Lemma \ref{Lemmaonenonneighbour} for example for $U=A$ and $W=A^c$. With this we can infect $l-r+f(l-r)-1$ vertices in $A^c$ and all vertices in $A$ which then infects all of $V(G)$.
\end{proof}
We will now show that this is tight.
\begin{Theorems} Let $l-r+2f(l-r)-2 \geq r \geq l-r+3$ and $n$ be divisible by $2r-l-1$. Let $H$ be a bipartite graph on parts $U$ and $W$ of size $\frac{2r-l-2}{2r-l-1} n$ and $\frac{1}{2r-l-1}$ such that every vertex in $U$ has degree $1$ and every vertex in $W$ has degree $2r-l-2$. Let $G$ be the graph consisting of $H$ together with two cliques on the vertices of $U$ and $W$. Then $G$ has no percolating set of size $l$ and $D(G)\geq n+2r-l-3$.  
\label{bigktightness}
\end{Theorems}
\begin{proof}
First we verify the degree conditions. Every vertex in $U$ has degree at least $\vert U \vert $ and every vertex in $W$ has degree at least $\vert W \vert +2r-l-3$. The only non-edges are between $U$ and $W$ and therefore the sum of any two non-adjacent vertices have $\deg(x)+\deg(y)\geq \vert U \vert+ \vert W \vert + 2r-l-3 =n+2r-l-3$. \\
For any percolating set we need to start with at least $r-1$ infected vertices in $U$ otherwise we will never be able to infect any other vertex in $U$ since every vertex has only $1$ neighbour in $W$ and we need to start with at least $r-l+2$ infected vertices in $W$ otherwise we will never be able to infect any other vertex in $W$. But then we need to initially infect at least $l+1$ vertices.
\end{proof}
Note that the above theorems imply for $l=r$ that $D_0(n,r)=n+r-2$ when $r\in \{3,4\}$ in the case when $n$ is large enough and satisfies the described divisibility conditions. Gunderson showed that the construction in Theorem $\ref{ThmTightness}$ has no percolating set if $r\geq 5$. This together with our results from Theorem $\ref{mainthm}$ and Theorem $\ref{ThmTightness}$ implies that $D_0(n,r)= n+ 2r-7$ when $r\geq 5$ and $n$ is large enough and even. Recall that Dairyko et al. \cite{dairyko2016ore} showed that $D_0(n,2)=n-1$ if $n\geq 6$. We therefore have a full result in the case when $l=r$ which is tight when $n$ satisfies some divisibility conditions. 
\begingroup
\def\theTheorems{\ref{Corollarykiszero}}
\begin{Cor}
Given $r\geq 1$, let $n=n(r)$ be sufficiently large. For $r\notin \{1,2,4\}$ let $n$ be even and for $r=4$ let $n$ be divisible by $3$, then
\begin{align*}
D_0(n,r)= \begin{cases} 
      n+2r-7 & \text{for } r\geq 5,\\ 
     n+ r-2 & \text{for } r \in \{3,4\}, \\ 
     n-1 & \text{for } r \in \{1,2\}.\\
   \end{cases} 
\end{align*}
\end{Cor}
\addtocounter{Theorems}{-1}
\endgroup
In general we have almost tight bounds if $2r\geq l+3$. Answering Gunderson's question about minimum degree conditions for percolating sets of size $l$ we get the following result.
\begin{Cor} Given $l\in \{r, \dots, 2r-3\}$, let $n=n(r,l)$ be sufficiently large. Then
\begin{align*}
\delta_0(n,r,l)= \begin{cases} 
      \floor*{\frac{n}{2}}+2r-l-f(l-r) & \text{if } 3r\geq 2l+2f(l-r)+4 \\ 
     \floor*{\frac{n}{2}}+ 2r-l-f(l-r)-1-d(n,r,l) & \text{if } 2(l-r) +f(l-r)+3 \geq r \\
     &\text{and } r\geq l-r+2f(l-r)-1 
        \end{cases} 
\end{align*}
where $d(n,r,l) \in \{0,1\}$. For $l-r+ 2 f(l-r)-2 \geq r \geq l-r+ 3$ we have the upper bound
\begin{align*}
\delta_0(n,r,l) \leq \frac{n+2r-l-2}{2}.
\end{align*}
\end{Cor}
\begin{proof}
This is a summary of Theorem \ref{mainthm}, \ref{ThmTightness}, \ref{bigkbound} and Corollary \ref{Corweak}. 
\end{proof}
\section{Concluding Remarks}
We have shown that if $3r \geq 2l +f(l-r)+4$, then $D_0(n,r,l-r)=n+ 4r-2l-2f(l-r)-1$, and if $l-r+ 2 f(l-r)-2 \geq r \geq l-r+ 3$, then $D_0(n,r,l)=n+2r-l-2$ under some divisibility conditions on $n$. If $n$ does not satisfy those divisibility conditions, can we improve the upper bound? Furthermore, it would be nice to get an exact bound on $D_0(n,r,l)$ in the case when $2(l-r) +f(l-r)+3 \geq r \geq l-r+2f(l-r)-1$ as we only know that $D_0(n,r,l)=n+ 4r-2l-2f(l-r)-1-j(n,r,l)$ where $j(n,r,l) \in \{0,1,2,3\}$ by Theorem \ref{mainthm} and Corollary \ref{Corweak}.

Moreover, one could find out whether $\delta_0(n,r,l)$ matches the bound from the Ore-type setting in the case when $l-r+ 2 f(l-r)-2 \geq r \geq l-r+ 3$. One can also see that in some cases the Ore-type condition is not an extension of the minimum degree condition as for $l=r$ we have shown that $D_0(n,4)=n+2$ which only implies $\delta_0(n,4) \leq \ceil*{\frac{n}{2}}+1$ but in fact Gunderson has shown that $\delta_0(n,4)=\floor*{\frac{n}{2}}+1$ so our result is not enough to show a tight upper bound on $\delta_0(n,4)$. It would be also interesting to get exact bounds if $l \geq 2r-2$ for the minimum degree setting and to determine if we can get a similar bound in the Ore-type setting.
\section{Acknowledgments}
I would like to express my very great appreciation to my supervisor Shagnik Das for his constant support and encouragement. Moreover, I would like to thank Ander Lamaison for discussions on Section \ref{ChapterExtensionGunderson} and Sebasti\'{a}n Gonz\'{a}lez Hermosillo de la Maza for helpful remarks on the writing.
\bibliographystyle{abbrv}
{\small
\bibliography{bibfile}}
\pagebreak
\newpage
\appendix
\appendixpage
\section{Proof of Lemma \ref{Lemmatightness}}
\begin{proof}[Proof of Lemma \ref{Lemmatightness}]
Before we start the proof, note that $ \vert N(u_1,\dots,u_j) \vert = \left( \sum_{i=1}^{j}  \vert N(u_i) \vert \right)- x$ where $x$ counts the number of times vertices $w \in N(u_1,\dots,u_j)$ were counted more often in $\sum_{i=1}^{j}  \vert N(u_i) \vert$ than in $\vert N(u_1,\dots,u_j) \vert $, i.e. $x=\sum_{w \in N(u_1,\dots,u_j)} x_w $ where $x_w= \left( \sum_{i=1}^j 1_{N(u_i)}(w) \right)-1$. \\
\ \\
We will first show that if $H$ has a cycle of size $2j\leq 2g$, then we can find $u_1, \dots, u_j$ such that $\vert N(u_1,\dots,u_j)\vert<\left( \sum_{i=1}^{j}  \vert N(u_i) \vert \right) -(j-1)$. We simply take the $u_i$ that appear in the cycle of size $2j$ namely let the cycle be $u_1w_1u_2\dots u_jw_ju_1$. Note that each $w_i$ has at least two neighbours in $u_1,\dots, u_j$. But then 
\begin{align*}
\sum_{w \in N(u_1,\dots,u_j)} x_w \geq \sum_{k=1}^j x_{w_k}\geq j.
\end{align*}
\ \\
Suppose now $H$ has no cycles of size at most $2g$. We will determine how large $N(u_1,\dots,u_j)$ is for any $u_1,\dots, u_j \in U$ by summing up $\sum_{i=1}^{j} \vert N(u_j) \vert$ and subtracting how often we overcounted a vertex in $N(u_1,\dots,u_j)$. A vertex $w$ in $N(u_1,\dots,u_j)$ is overcounted exactly $\left( \sum_{i=1}^{j} 1_{N(u_i)}(w) \right) -1$ times and let $u(w)$ be the smallest indexed $u_i$ such that $w\in N(u_i)$. We build an auxiliary graph $\tilde{G}$ of vertices $\tilde{u}_1, \dots, \tilde{u}_{j}$ and let $\tilde{u}(w)$ have the same index as $u(w)$. For every $w$ connect $\tilde{u}_i$ to $\tilde{u}(w)$ if and only if $w \in N(u_i)$. Note that $\tilde{G}$ is a simple graph since if we had two edges between $u_i\tilde u_j$ we would have $w_1 \neq w_2\in N(u_1)\cap N(u_j)$ but this is not possible as there are no $4$-cycles in $H$. We constructed $\tilde{G}$ such that an edge in $\tilde{G}$ corresponds to an overcounting of a vertex in $N(u_1,\dots,u_j)$. Suppose $\tilde{G}$ contains a cycle $\tilde{v}_1, \tilde{v}_2, \dots, \tilde{v}_l,\tilde{v}_1$ with corresponding vertices $v_1,\dots,v_l \in \{u_1,\dots,u_{j}\}$. Then there exists for each $i$ a vertex $w_i$ in $W$ such that $w_i \in N(v_i) \cap N(v_{i+1})$ and either $v_i=u(w_i)$ or $v_{i+1}=u(w_i)$. Note that if $w_i=w_j$, then $v_i=u(w_i)$ or $v_{i+1}=u(w_i)$ and $v_j=u(w_i)$ or $v_{j+1}=u(w_i)$. Since $u(w_i)$ can appear only once in the cycle $v_1,\dots,v_l$ we either have $i=j$ or $i=j+1$ or $i=j-1$. Consider the walk $v_1,w_1,v_2,\dots, v_l,w_l$ where $l\geq 3$. If a vertex $w\in W$ appears twice, then we delete the vertex $u(w)$ that appears between the $w$'s and we also delete one $w$. Note that in the resulting walk we can have only one $w$ left. If we continue like this, we get a cycle of length at most $2j$ in $H$ which is a contradiction. This shows that $\tilde{G}$ is acyclic and cannot have more than $j-1$ edges. Therefore we overcounted at most $j-1$ times a vertex in $N(u_1,\dots,u_j)$ and $\vert N(u_1,\dots,u_j) \vert \geq \left( \sum_{i=1}^{j}  \vert N(u_i) \vert \right)-\left( j-1\right)$ which gives the desired result.\\
\end{proof}
\section{Proof of Proposition \ref{PropTightness}}
\begin{proof}[Proof of Proposition \ref{PropTightness}]
If $2r=l+f(l-r)$, we can take an empty graph and if $2r=l+f(l-r)+1$ we can take a matching between two parts $U$ and $W$ of size $n$ each. Erd\H{o}s and Sachs \cite{erdos1963regulare} constructed for all $s,t \geq 2$ and $m=m(s,t)$ large enough graphs that are of size $2m$ and $s$-regular and that have girth exactly $t$ but they do not need to be bipartite. We choose $s=2r-l-f(l-r)$ and $t=2f(l-r)+2$ to get a $2r-l-f(l-r)$-regular graph $G$ on vertex set $[2m]$ and girth $2f(l-r)+2$. From that we can obtain a $2r-l-f(l-r)$-regular bipartite graph $\tilde{G}$ of size $4m$ and girth at least $2 f(l-r)+2$ on parts $U=\{u_1, \dots, u_n\}$ and $W=\{w_1, \dots, w_n\}$ and $u_i \sim w_j$ if and only if $i \sim j$. Note that a cycle $\tilde{C}$ in $\tilde{G}$ either corresponds to a cycle in $G$ or to a closed walk in $G$ but since a closed walk in $G$ contains a cycle we know that $\tilde{C}$ needs to have size at least $2f(l-r)+2$ which shows that $\tilde{G}$ has girth at least $2f(l-r)+2$.

We want to show now that we can in fact for $n$ large enough always find a bipartite graph of size $2n$ with the mentioned degree and girth conditions. We take $m$ to be large enough such that we get $\tilde{G}$ on parts $U$ and $W$ with at least $n_0 =(s-1)\sum_{i=1}^{f(l-r)} s(s-1)^{2i}+s$.

Given an $s$-regular bipartite graph of size $2n\geq 2  n_0$ with girth at least $2f(k)+2$ we want to construct an $s$-regular bipartite graph of size $2n+2$ with girth at least $2f(k)+2$. Given a $s$-regular bipartite graph $G$ of size $n$ we take a set of $s$ vertices $u_1, \dots, u_{s}$ in $U$ such that their distance is pairwise at least $2f(l-r)+2$. Note that this is possible since for each $u_i$ there are at most $\sum_{i=1}^{f(l-r)} s(s-1)^{2i}$ vertices of distance at most $2f(l-r)+1$ and we chose $n_0$ sufficiently large. We pick a neighbour $w_i$ for each $u_i$ and $1\leq i \leq s$ and note that $w_1, \dots, w_{s}$ are pairwise different since there are no $4$-cycles in $G$. We delete the edge between $u_i$ and $w_i$ for each $i$ and add a new vertex $u$ to $U$ which we connect to all $w_i$. Similarly we add a new vertex $w$ to $W$ and connect it to all $u_i$. Note that this gives us an $s$-regular graph with $2$ more vertices. We show now that it has girth at least $2f(l-r)+2$.

Suppose we created a cycle of size $2j\leq 2f(l-r)$. Note that this cycle needs to contain either $u$ or $w$ or both. Let us assume that the cycle contains only $u$ and let without loss of generality $w_1$ and $w_2$ be the vertices to which $U$ is adjacent in the cycle. Observe that this means that $w_1$ and $w_2$ are of distance at most $2j-2$ in $G$ and therefore $u_1$ and $u_2$ of distance at most $2j$ in $G$ which is a contradiction. A similar reasoning applies if only $w$ was contained in the cycle. Now suppose both, $u$ and $w$ are contained in the cycle. Take one of the paths from $u$ to $w$ in the cycle which is without loss of generality $u w_1 v_1 v_2 \dots v_k u_2 w$. Note that this means that $w_1$ and $u_2$ are of distance at most $2f(l-r)-3$ in $G$ and therefore $u_1$ and $u_2$ of distance at most $2f(l-r)-2$ in $G$ which is a contradiction. 
\end{proof}
\begin{Remark}
Instead of using the result of Erd\H{o}s and Sachs, we could have used a later result of Füredi et al. \cite{furedi1995graphs} who showed that there are bipartite graphs $G$ with bidegree $(s,t)$ for any $s,t$ with girth exactly $2m$ for any $m \geq 2$. 
\end{Remark}
\end{document}